\documentclass[a4paper,10pt,reqno]{amsart}

\usepackage{hyperref}
\usepackage{amssymb}
\usepackage{eucal}
\usepackage{eufrak}
\usepackage{graphics}
\usepackage{graphicx}
\usepackage{epsfig}
\usepackage{multicol}
\usepackage{xypic}
\usepackage{amsmath}
\usepackage{wasysym}
\usepackage{dsfont}
\usepackage{float}

\usepackage{eufrak}
\usepackage{graphics}
\usepackage{graphicx}
\usepackage{epsfig}
\usepackage{multicol}
\usepackage{xypic}
\usepackage{amsmath}
\usepackage{color}

\newtheorem{proposition}{Proposition}[section]
\newtheorem{definition}{Definition}[section]

\newtheorem{remark}[definition]{Remark}

\newtheorem{theorem}[definition]{Theorem}
\newtheorem{algorithm}[definition]{Algorithm}

\newtheorem{corollary}[definition]{Corollary}

\newcommand{\lp}{\left(}
\newcommand{\rp}{\right)}
\newcommand{\lc}{\left\{}
\newcommand{\rc}{\right\}}
\newcommand{\der}{\partial}

\newcommand{\bra}{\langle}
\newcommand{\ket}{\rangle}
\newcommand{\R}{\mathds{R}}      
\newcommand{\Z}{\mathds{Z}}      

\newcommand{\I}{\mathds{I}}
\newcommand{\Flder}{\rightarrow}

\usepackage{amssymb}

\newcommand{\proa}{A^*G \mbox{$\;$}_{\tau^*} \kern-3pt\times_\alpha
G \mbox{$\;$}_\beta \kern-3pt\times_{\tau^*} A^*G}

\newcommand{\alg}{\mathfrak{so}(3)}

\newcommand{\ca}{\mbox{cay}}
\newcommand{\ad}{\mbox{ad}}
\newcommand{\om}{\omega}

\newcommand{\al}{\mathfrak{g}}
\newcommand{\dal}{\mathfrak{g}^{*}}

\newcommand{\ald}{\mathfrak{g}^{_{D}}}
\newcommand{\algd}{\mathfrak{so}(3)^{_{D}}}

\newcommand{\aldp}{\mathfrak{g}^{_{D_{\epsilon}}}}
\newcommand{\algdp}{\mathfrak{so}(3)^{_{D_{\epsilon}}}}
\newcommand{\Ese}{\mathcal{S}}

\begin{document}

\title{On some aspects of the discretization of the Suslov problem}

\author[F. Jim\'enez]{Fernando Jim\'enez}
\address{F. Jim\'enez: Zentrum Mathematik der Technische Universit\"at M\"unchen, D-85747 Garching bei M\"unchen, Germany} \email{fernando.jimenez.alburquerque@gmail.com}

\author[J. Scheurle]{J\"urgen Scheurle}
\address{J. Scheurle: Zentrum Mathematik der Technische Universit\"at M\"unchen, D-85747 Garching bei M\"unchen, Germany} \email{scheurle@ma.tum.de}

\thanks{This research was supported by the DFG Collaborative Research Center TRR 109, ``Discretization in Geometry and Dynamics''. 
}

\keywords{Nonholonomic mechanics, discretization as perturbation, geo\-me\-tric integration, discrete variational calculus, Lie groups and Lie algebras, reduction of mechanical systems with symmetry}

\subjclass[2000]{34C15; 37J15; 37N05; 65P10; 70F25.}

\maketitle

\begin{abstract}
In this paper we explore the discretization of Euler-Poincar\'e-Suslov equations on $SO(3)$, i.e.  of the Suslov problem. We show that the consistency order corresponding to the unreduced and reduced setups, when the discrete reconstruction equation is given by a Cayley retraction map, are related to each other in a nontrivial way. We give precise conditions under
which general and variational integrators generate a discrete 
flow preserving the constraint distribution. We  establish general consistency bounds and illustrate the performance of several discretizations by some plots. Moreover, along the lines of \cite{JiSch} we show that any
constraints-preserving discretization may be understood as being generated by the exact evolution map of a time-periodic non-autonomous perturbation of the
original continuous-time nonholonomic system. 
\end{abstract}

\section{Introduction}

The Lagrangian formulation of mechanical systems with nonholonomic con\-straints
has been extensively studied in recent years (see \cite{Bl, Manolo} for a complete description and extensive bibliographies). In short, this kind of systems
are cha\-rac\-te\-ri\-zed by so-called nonholonomic constraints, i.e. constraints involving
both configuration as well as velocity variables, and which can not be integrated
to purely configuration-dependent constraints (in this case the constraints
are called holonomic). Moreover, the preservation of certain structural properties of mechanical systems  by sui\-ta\-ble integrators is a de\-li\-ca\-te issue upon which a lot of attention has been put by the Geometric Mechanics comunity (see the recent works, such as \cite{CoMa,SDD1,SJD,IMMM,KFMD,McPer}, which have introduced numerical integrators for
holonomic systems with very good energy behavior and various  preservation properties). The approaches in most of these references are based on the ideas of \cite{MarsdenWest,Mose}. In these works, the continuous variational principles are replaced by discrete ones aiming to obtain proper integrators approximating the continuous dynamics. We will call the integrators related to this framework {\it variational integrators}. Analogously, in the case of nonholonomic mechanics, the continuous Lagrange-d'Alembert's principle, which provides the actual dynamics, is replaced by a
discrete Lagrange-d'Alembert's principle on the discrete phase space. Of special interest are the seminal works on nonholonomic integration \cite{CoMa,McPer}, where a discrete version of the Lagrange-d'Alembert principle is proposed by introducing a proper discretization of the nonholonomic distribution.
In \cite{JiSch}, the focus has been on numerical integrators including variational integrators that 
exactly respect the original continuous constraint distribution. Also, corresponding consistency estimates were derived. In particular, it has been shown that any integrator that preserves the original continuous constraint distribution may be understood as being generated by the exact evolution map of a 
time-periodic non-autonomous perturbation of the continuous-time nonholonomic system, where 
the size of the perturbation is related to the order of consistency of the integrator.

Reduction theory is one of the
fundamental tools in the study of mechanical systems with symmetry and it concerns the removal of symmetries using the associated
conservation laws. A Lagrangian system is called symmetric w.r.t.  a Lie group if the Lagrangian function is invariant under the tangent lift of the action of the Lie group on the configuration manifold. Furthermore, for a symmetric mechanical system the process of reduction eliminates the directions along the group variables and thus provides a system with fewer degrees of freedom.  When the configuration ma\-ni\-fold is the Lie group itself,  for unconstrained Lagrangian systems   the process of reduction leads from the Euler-Lagrange equations to the Euler-Poincar\'e equations. Needless to say, nonholonomic systems may also possess symmetries, and the geometrical treatment of such situations has been studied in \cite{BlKrMaMu1996,CLMD} among other references. On the other hand, variational integrators for reduced systems were carefully studied from the theoretical point of view in \cite{MPS, MPS2}. The combination of these two issues, i.e. variational integrators in the context of  symmetric nonholonomic Lagrangian systems, has been addressed  in \cite{FeZen} in the case of a {\it ll system}, i.e. a nonholonomic system whose configuration manifold is a Lie group $G$, and where both the Lagrangian and the constraint distribution are invariant with respect to the induced left action of $G$ on $TG$. Here,the dynamics of the reduced  system is described by the {\it Euler-Poincar\'e-Suslov} equations \cite{Koz,Suslov}. Other interesting discretizations of symmetric nonholonomic systems can be found in \cite{Fedorov}.

The Euler-Poincar\'e-Suslov equations on $SO(3)$, i.e. the Suslov problem \cite{Suslov}, are the object of study in the present  paper, more concretely from the point of view of the results presented in \cite{JiSch}. Particularly, we focus on their discretization  and meticulously study  conditions under which the constraints are preserved by the corresponding integrator.  Both, general as wll as variational intgrators
will be considered. We study the relationship between the orders of consistency of integrators in the
unreduced and reduced setting, that are related by reconstruction and reduction, respectively. 
We find out that the relationship is nontrivial, and we figure it out precisely for the case of  a particular retraction map relating $SO(3)$ and $\alg$, namely the Cayley map. The study of consistency  requires the introduction of a suitable metric on $SO(3)$.  We present numerical results for several integrators
and carefully study their order of consistency. 
Finally, we address the issue of discrepancies between the continuous and the discrete dynamics 
for the Suslov problem based on   previous results in \cite{FiSch,JiSch}.
\medskip

The paper is structured as follows: In \S\ref{Preli}, \S\ref{Symm} and \S\ref{VIRS} we provide a comprehensive introduction to nonholonomic mechanics, variational nonholonomic integrators, symmetries of  nonholonomic Lagrangian systems, the Euler-Poincar\'e-Suslov equations and, finally, variational integrators for reduced systems. We provide a new version of the discrete Euler-Poincar\'e-Suslov equations, defined on the Lie algebra through the incorporation of a retraction map, in corollary \ref{CoroRetr}. \S\ref{Principal} establishes the link between the results in \cite{JiSch} and the reduced framework for $SO(3)$, i.e. we perform an exhaustive study of the Suslov problem. First, \S\ref{Consistenci} is devoted to the study of consistency orders as mentioned above.
Then,  in subsection \S\ref{dreps}, we study general discretizations, as well as sufficient conditions in the case of variational integrators that guarantee  preservationof the continuous-time constraint 
distribution. In \S\ref{GeneralDisc}, we give consistency results both for general and variational discretizations, and illustrate their performance through some plots.
Finally, \S\ref{Pertur} is devoted to the study of distribution-preserving discretizations of the Euler-Poincar\'e-Suslov problem viewed as perturbation of the continuous dynamics. Here, we apply the result by Fiedler and Scheurle \cite{FiSch} to the system under consideration, and establish a bound of the perturbation produced in the unreduced dynamics by a reduced integrator, and conversely.
\medskip

Throughout the paper, we use Einstein's convention for the summation over
repeated indices unless the opposite is stated.

\section{Preliminaries}\label{Preli}

\subsection{The nonholonomic setting}\label{NH}
We define a nonholonomic Lagrangian system on a Lie group as a triple $(G,L,D)$, where $G$ is the Lie group, $L:TG\Flder\R$ is the Lagrangian function and $D\subset TG$ is a {\it constraint distribution}, which we assume to be a constant rank {\it linear vector subbundle} of $TG$ and non-integrable in the Frobenious sense.  Locally, the  constraints are written as follows:
\begin{equation}\label{LC}
\phi^{\alpha}\lp g,\dot g\rp=\mu^{\alpha}_{i}\lp g\rp\dot g^{i}=0,\hspace{2mm} 1\leq \alpha\leq m,
\end{equation}
where $(g^i,\dot g^i)$, $i=1,...,n$, are  coordinates of $TG$, i.e. the constraints are linear w.r.t. the 
velocity variables. The annihilator of $D^{\circ}$ is locally given by
\[
D^{\circ}=\mbox{span}\lc\mu^{\alpha}=\mu_{i}^{\alpha}(g)\,dg^{i};\hspace{1mm} 1\leq \alpha\leq m\rc,
\]
where the one-forms $\mu^{\alpha}$ are assumed to be  linearly independent, i.e.  $\mbox{rank}\lp D\rp=n-m$, with $m<n$.

In addition to the distribution, we need to specify the dynamical evolution of the system through the Lagrangian function. In nonholonomic mechanics, the procedure leading from the Newtonian point of view to the Lagrangian one is given by the Lagrange-d'Alembert's principle. This principle says that a
continuously differentiable 
 curve $g:I\subset \R\Flder G$  describes an admissible motion of the system if
\[
\delta\int^{t_2}_{t_1}L\lp g\lp t\rp, \dot g\lp t\rp\rp dt=0
\]
with respect to all variations such that $\delta g\lp t\rp\in D_{g\lp t\rp}$, $t_1\leq t\leq t_2$,  the fixed endpoint condition is satisfied, and  the velocity of the motion  satisfies the constraints. Using Lagrange-d'Alembert's principle, we arrive at the nonholonomic equations, which in coordinates read 

\begin{subequations}\label{LdAeqs}
\begin{align}
\frac{d}{dt}\lp\frac{\der L}{\der\dot g^{i}}\rp-\frac{\der L}{\der g^{i}}&=\lambda_{\alpha}\,\mu^{\alpha}_{i}(g),\label{Con-1}\\\label{Con-2}
\mu_{i}^{\alpha}(g)\,\dot g^{i}&=0,
\end{align}
\end{subequations}
where $\lambda_{\alpha},\hspace{1mm} \alpha=1,...,m$ are ``Lagrange multipliers''. The right-hand side of equation (\ref{Con-1}) represents the reaction forces due to the constraints, and equations (\ref{Con-2}) represent the constraints themselves. Needless to say, under appropriate regularity conditions the equations \eqref{LdAeqs} generate a local flow within $D\subset TG$. For more details we refer to \cite{Arn,Bl}.

In the present  paper we are interested in $D$-preserving  discretizations of the solutions of  \eqref{LdAeqs}. We clarify this concept before proceeding.

\begin{definition}\label{Dpreserv}
Consider a sequence of points $\lc v_{g_k}\rc_{0:N}\in T_{g_k}G$, i.e. $\tau_G(v_{g_k})=g_k $ with $\tau_G:TG\Flder G$ being the canonical projection,  $k=0,...,N$. We say that this sequence is $D$-preserving if $v_{g_k}\in D_{g_k}$ for every $k$ (in other words, $\mu^{\alpha}(g_k) \dot g_k=0$ in  local coordinates $(g_k,\dot g_k)$ of $v_{g_k}$).
\end{definition}
The integer $N$ will be the number of steps  of a given integrator. This number is related to the time-step $\epsilon$ of the integrator  and the time interval $t_2-t_1$ by $N=(t_2-t_1)/\epsilon$.

\subsection{Nonholonomic integrators}

Discretizations of the Lagrange-d'A\-lem\-bert  principle for Lagrangian systems with nonholonomic cons\-traints have been introduced in \cite{CoMa,McPer} as a nonholonomic extension of variational integrators  (see
\cite{Hair,MarsdenWest,Mose}). To define a discrete nonholonomic system providing a discrete flow on a submanifold of $G\times G$ one needs three ingredients: a discrete Lagrangian, the constraint distribution $D\subset TG$ and a discrete constraint space $D_d\subset G\times G$. 

\begin{definition}\label{DiscNHSet} 
A discrete nonholonomic system is given by the quadruple $(G\times G,L_d,D_d, D)$, where:
\begin{enumerate}
\item $D_d$ is a submanifold of $G\times G$ of dimension $2n-m$ with the additional property that
\[
I_d=\lc (g,g)\,|\,g\in G\rc\subset D_d.
\]

\item $L_d:G\times G\Flder\R$ is the discrete Lagrangian, which is chosen as an appropriate approximation of the action integral w.r.t. one time step, i.e. $L_d(g_0, g_1)\approx \int^{t_0+\epsilon}_{t_0} L(g(t), \dot{g}(t))\; dt.$
\end{enumerate}
\end{definition}
We define the discrete Lagrange-d'Alembert principle (DLA) to be the extremization of the action sum 
\begin{equation}\label{DiscAC}
S_d=\sum_{k=0}^{N-1}L_d(g_k,g_{k+1})
\end{equation}
among all sequences of points $\lc g_k\rc_{0:N}$ with given fixed end points $g_0,g_N$, where the variations must satisfy $\delta g_k\in D_{g_k}$ (in other words $\delta g_k\in \mbox{ker}\,\mu^{\alpha}$) and $(g_k,g_{k+1})\in D_d$ for all $k\in\lc 0,...,N-1\rc$. This leads to the following set of discrete nonholonomic equations
\begin{subequations}\label{DLA}
\begin{align}
D_1L_d(g_k,g_{k+1})&+D_2L_d(g_{k-1},g_{k})=\lambda_{\alpha}\,\mu^{\alpha}(g_k),\label{DLAa}\\
(g_k,g_{k+1})&\in D_d\label{DLAb}.
\end{align}
\end{subequations}
For the sake of clarity, the condition \eqref{DLAb} may be rewritten as $\phi^{\alpha}_d(g_k,g_{k+1})=0$, where $\phi^{\alpha}_d:G\times G\Flder\R$ is the set of $m$ functions whose annihilation defines $D_d$. 
Equations \eqref{DLA}, where $\lambda_{\alpha}=(\lambda_{k+1})_{\alpha}$ is chosen appropriately by projecting onto $D_d$, define a local {\it discrete nonholonomic flow map} $F_{L_d}^{nh}:D_d\Flder D_d$ given by
\begin{equation}\label{NHFlow}
F_{L_d}^{nh}(g_{k-1},g_k)=(g_k,g_{k+1}),
\end{equation} 
where  $(g_{k-1},g_k)\in D_d$ and $g_{k+1}$ satisfyes \eqref{DLA}. Here, we assume  a regularity condition, say the matrix 
\[
\lp
\begin{array}{cc}
D_1D_2L_d(g_k,g_{k+1}) & \mu^{\alpha}(g_k)\\\\
D_2\phi^{\alpha}_d(g_k,g_{k+1}) &0
\end{array}
\rp
\]
is regular, to be  fullfiled for each $(g_k,g_{k+1})$ in a neighborhood of the diagonal of $G\times G$ (see \cite{CoMa} for further details).

\begin{remark}\label{RmkMulti}
{\rm In this paper, we approximate time-continuous, nonholonomic systems by discrete nonholonomic systems. Note that, throughout the paper, according to \cite{CoMa} and \cite{FeZen} (leading to equations \eqref{DLA} and \eqref{DEPEq}, respectively) we always choose the constraint distribution of the discretized system to be the constraint distribution of the original, time-continuous system. Also, we choose the same basis $\mu^\alpha$ ($\alpha$ = $1, . . . ,m$) for the annihilator of the constraint distribution in both cases. Thus, the scalar-valued Lagrange multipliers $\lambda_\alpha$ represent corresponding components (coordinates) of the reaction forces in both cases. In particular, this allows to estimate the approximation error for the reaction force in terms of the approximation errors for the Lagrange multipliers. Of course, this is still possible, when corresponding elements of the two bases differ by a distance,  the order of which with respect to the size of the time-step $\epsilon$ of the discretization is sufficiently high. The resulting reaction forces do not depend on the choice of those bases. In fact, they are even defined globally on $G$ by the Lagrange multipliers $\mu^\alpha$, since one can always choose a trivialization of the tangent bundle $TG$ of a Lie group $G$ (as we do below).}
\end{remark}

To obtain distribution preserving integrators (in the sense of Definition \ref{Dpreserv}) within this setting is not straightforward, since the equations \eqref{DLA} are defined on the discrete space $G\times G$ and moreover determine a discrete flow on the discrete distribution $D_d$. Therefore, one needs to locally relate $(g_k,g_{k+1})\in G\times G$ to $v_{g_k}\in TG$, a task accomplished in \cite{JiSch} by means of finite difference maps $\rho$ \cite{McPer}.

\begin{definition}\label{FDM}
A {\rm finite difference map} $\rho$ is a diffeomorphism $\rho:U(I_d)\Flder V(Z)$, where $U(I_d)$ is a neighborhood of the diagonal $I_d$ in $G\times G$, and $V(Z)$ denotes a neighborhood of the zero section of $TG$, i.e $Z:G\Flder TG$ s.t. $Z(g)=0_g\in T_gG$, which satisfies the following conditions:
\begin{enumerate}
\item $\rho(I_d)=Z$,
\item $\tau_G\circ\rho(U(I_d))=G$, where $\tau_G:TG\Flder G$ is the canonical projection,
\item $\tau_G\circ\rho|_{I_d}=\pi_1|_{I_d}=\pi_2|_{I_d}$, where $\pi_1$ and $\pi_2$  are the projections from $G\times G$ to its first and second component $G$, respectively.
\end{enumerate}
\end{definition}
For any finite difference map $\rho$, the so-called {\it velocity nonholonomic integrator} 
\begin{equation}\label{vnhi}
\tilde F_{L_d}:=\rho\circ F_{L_d}^{nh}\circ\rho^{-1},
\end{equation}
defines a flow $\tilde F_{L_d}:T_{\tilde g_k}G\Flder T_{\tilde g_{k+1}}G$; $v_{\tilde g_k}\mapsto v_{\tilde g_{k+1}}$ s.t. $\tau_G(v_{\tilde g_k})=\tilde g_k$ and $\tau_G(v_{\tilde g_{k+1}})=\tilde g_{k+1}$. In general, this flow is not $D-$preserving in the sense of Definition \ref{Dpreserv}, but sufficient conditions for that to hold are given in \cite{JiSch}, which in general require   the  discrete constraints to be given by $\phi_d^{\alpha}$, as well as  an appropriate redefinition of the discrete nodes $\lc g_k\rc\mapsto \lc \tilde g_k\rc$. The result is stated in the following proposition: 
\begin{proposition}\label{Preservation}
Assume that $v_{\tilde g_k}\in D_{\tilde g_k}$. If $D_d$ is defined by $\phi^{\alpha}_d:=\mu^{\alpha}\circ\rho:G\times G\Flder\R$, then  $ v_{\tilde g_{k+1}}$ defined by the velocity nonholonomic integrator \eqref{vnhi}, i.e. $\tilde F_{L_d}$, belongs to
$D_{\tilde g_{k+1}}$. In other words, $\tilde F_{L_d}: D_{\tilde g_k}\Flder D_{\tilde g_{k+1}}$, i.e. it generates a D-preserving sequence in the sense of Definition \ref{Dpreserv}, provided that $v_{\tilde g_0} \in D_{\tilde g_0}$ holds.
\end{proposition} 
Therefore, the functions $\phi^{\alpha}_d$ prescribed in Proposition \ref{Preservation} may be considered as a suitable discretization of the nonholonomic constraints \eqref{LC}.

\section{Symmetries}\label{Symm}

\subsection{The Euler-Poincar\'e-Suslov equations}\label{LLSystems}
We say that the Lagrangian is {\rm invariant} under a group action $\Phi:G\times G\Flder G$, if $L$ is invariant under the lifted action of $G$ on $TG$, i.e.
$L\circ T\Phi_g=L.$ Such a symmetry allows to define a {\it reduced Lagrangian} on the reduced phase space 
$TG/G=(G\times\al)/G\cong\al$, say $l:\al\Flder\R$.  We have employed  in the previous relation the left trivialization of $TG$ represented by  the mapping $\mbox{tr}:TG\Flder G\times\al$, $v_g\mapsto\lp g,\,T_g\ell_{g^{-1}}v_g\rp$, where  $\ell:\,G\times G\Flder G;\, (g,h)\mapsto g\,h$; $\ell_g:G\Flder G$, $h\mapsto g\,h$, is the left action and $\al$ is the Lie algebra of $G$.
The reduced Lagrangian is obtained as follows using the symmetry:
\begin{equation}\label{RedLagra}
L(g,\dot g)=L(g^{-1}\,g,g^{-1}\,\dot g)=L(e,\xi)=:l(\xi)
\end{equation}
Here we are employing the shorthand notation $T_g\ell_h v_g=:h\,\dot g$ for $v_g\in T_gG$ with coordinates $(g,\dot g)$, $e\in G$ is the identity element and $g^{-1}\,\dot g:=\xi\in\al$ is called the {\it reconstruction equation}.

In the nonholonomic case, besides a left-invariant $L$ we shall consider as well a left-invariant distribution $D$. Both ingredients account for a left-left or ll system. Here $D\subset TG$ is left-invariant if and only if there exists a subspace $\ald\subset\al$ such that $D_g=T_e\ell_g\ald\subset T_gG$ for any $g\in G$. Let $a^{\alpha}\in\dal$, $\alpha=1,...,m$, be  a basis of the annihilator of the subspace $\ald$, i.e.
\begin{equation}\label{alD}
\ald=\lc\xi\in\al\,|\,\bra a^{\alpha},\xi\ket=0\,,\alpha=1,...,m\rc.
\end{equation}
Consequently, the left-invariant constraints on $TG$ are defined by the equations $\bra a^{\alpha}\,,\,g^{-1}\dot g\ket=0$. 
This last equation establishes the correspondence between the nonholonomic constraints of a left-left system and the  ones  in \eqref{LC}. 

According to \cite{Koz}, the  dynamics of a ll system is determined by the so-called {\it Euler-Poincar\'e-Suslov} equations :
\begin{equation}\label{EuPoiSus}
\begin{split}
\frac{d}{dt}\lp\frac{\der l}{\der\xi}\rp&=\ad^*_{\xi}\lp\frac{\der l}{\der\xi}\rp+\lambda_{\alpha}a^{\alpha},\\
\bra a^{\alpha},\xi\ket&=0.
\end{split}
\end{equation}
Under certain regularity conditions, these equations provide the solution curve $\xi(t)\subset \ald$. Furthermore, the solution curve w.r.t.  the group variables $g(t)$ is obtained through the reconstruction equation $\dot g=g\,\xi$.

Analogously to Definition \ref{Dpreserv}, we now define  the notion of $\ald-$preservation.

\begin{definition}\label{gDpreserv}
Consider the sequence of points $\lc \xi_k\rc_{0:N}\in \al$,  $k=0,...,N$. We say that this sequence is $\ald$-preserving if $\xi_k\in\ald$ for every $k$. In other words, $\bra a^{\alpha},\xi_k\ket=0$ w.r.t.  the local representation of the nonholonomic reduced constraints \eqref{alD}.
\end{definition}

We observe that, taking into account  the symmetries  of the proposed problem and using the left trivialization (i.e. $v_{g_k}=(g_k,\xi_k)$), $D-$preservation in the sense of Definition \ref{Dpreserv} implies $\ald-$preservation in the sense of Definition \ref{gDpreserv}. The converse is not true in general, since it requires  to determine the sequence $\lc g_k\rc_{0:N}$ from the sequence $\lc \xi_k\rc_{0:N}$. This process is not trivial and may be called {\it discrete reconstruction}. We shall see that the variational procedure provides a possible approach for it.

\section{Discretization of ll systems}\label{VIRS}

The discretization of nonholonomic ll systems in accordance with the DLA algorithm \eqref{DLA} is thoroughly considered in 
\cite{FeZen}. The authors proposed a discretisation scheme under the natural assumptions that both the 
 discrete Lagrangian $L_d:G\times G\Flder\R$ and the discrete constraint space  $D_d\subset G\times G$ are invariant
 under the diagonal action of $G$ on $G\times G$ by left multiplication. We briefly recall this construction and its consequences.
 
 By invariance of $L_d$, one can define a {\em reduced discrete Lagrangian} $l_d:G\Flder \R$ by the rule 
 \begin{equation*}
L_d(g_k,g_{k+1})=L_d(e,g_{k}^{-1}g_{k+1}) =: l_d(W_k)
\end{equation*}
where $W_k:=g_{k}^{-1}g_{k+1}\in G$ is the  {\it left incremental displacement}. One should  interpret $W_k\in G$ as a finite difference approximation
on the group  of the infinitesimal  velocity $\xi =g^{-1}\dot g$ that belongs to the Lie algebra. The relation
$g_{k+1}=g_kW_k$  is   the discrete counterpart of the reconstruction equation $\dot g=g\xi$ in this scenario (how the equation $g_{k+1}=g_kW_k$ relates to the algebra elements will be  discussed below). 
 
 Similarly, by left invariance of $D_d$ there exists a {\em discrete displacement subvariety} $\mathcal{S} \subset G$ determined
 by the condition
 \begin{equation*}
(g_k,g_{k+1})\in D_d \qquad \mbox{if and only if} \qquad W_k=g_{k}^{-1}g_{k+1}\in \mathcal{S}.
\end{equation*}
This leads to  the definition of functions $\varphi_d^{\alpha}:G\Flder\R$ whose annihilation determines $\Ese$, namely
\begin{equation}\label{DiscDist}
\phi_d^{\alpha}(g_k,g_{k+1})=\phi_d^{\alpha}(g_k^{-1}\,g_k, g_k^{-1}\,g_{k+1})=\varphi_d^{\alpha}(W_k) = 0,\; \; \forall\, W_k\in\Ese,\,\,\alpha=1,...,m.
\end{equation}
Finally, the {\it reduced action sum} is given by
\begin{equation}\label{DiscACRed}
s_d=\sum_{k=0}^{N-1}l_d(W_k).
\end{equation}
The discrete counterpart of the Euler-Poincar\'e-Suslov equations \eqref{EuPoiSus} is established in the next theorem, presented in \cite{FeZen}, which is an extension of the results in \cite{BoSu} and \cite{MPS} concerning the discrete version of the classical Euler-Poincar\'e reduction.
\begin{theorem}\label{RedTheo}
Let $L_d:G\times G\Flder\R$ be a left-invariant discrete Lagrangian, $l_d:G\Flder\R$ be the reduced discrete Lagrangian, and $D\subset TG$, $D_d\subset G\times G$ be the constraint distribution (also left-invariant) and the discrete constraint submanifold respectively. Then the following assertions are equivalent:
\begin{enumerate}
\item $\lc g_k\rc_{0:N}$ is a critical point of the action \eqref{DiscAC} for constrained variations $\delta g_k\in D_{g_k}$ s.t. $\delta g_0=\delta g_N=0$.
\item $\lc g_k\rc_{0:N}$ satisfies the discrete nonholonomic equations \eqref{DLA}.

\item $\lc W_k\rc_{0:N-1}$ is a critical point of the reduced action sum \eqref{DiscACRed}, with respect to variations $\delta W_k$ induced by the constrained variations $\delta g_k\in D_{g_k}$.
\item $\lc W_k\rc_{0:N-1}$ satisfies the {\rm reduced nonholonomic equations} or {\rm discrete Euler-Poincar\'e-Suslov equations}:
\begin{subequations}\label{DEPEq}
\begin{align}
-r^{*}_{W_{k+1}} l^{\prime}_d(W_{k+1})+\ell_{W_{k}}^{*}l^{\prime}_d(W_{k})&=\lambda_{\alpha}\,a^{\alpha},\label{DEPEqa}\\ 
\varphi^{\alpha}_d(W_{k+1})&=0,\label{DEPEqb}
\end{align}
\end{subequations}
for $k=0,...,N-2$, where again $\lambda_{\alpha}=(\lambda_{k+1})_{\alpha}$ are chosen appropriately.
\end{enumerate} 
\end{theorem}

\subsection{Discretization using natural coordinate charts}\label{naturalCharts}
One observes that  \eqref{DEPEq} generates a discrete evolution in $G$ (more concretely in $\Ese$) while \eqref{EuPoiSus} generates a continuous evolution in $\al$ (more concretely in $\ald$).  A possible relationship between $G$ and $\al$  is achieved by means of a so-called
{\it retraction map} $\tau:\al\Flder G$: a local analytic 
diffeomorphism near the identity such that $\tau(\xi)\tau(-\xi)=e$,
where $\xi\in\mathfrak g$. Such a $\tau$ provides a natural coordinate chart on $G$, and the $W_k$ are regarded as small displacements on
the Lie group, linking $g_k$ and $g_{k+1}$. Thus, it is possible to express each $W_k$  through a Lie
algebra element 
\[
\xi_k=\tau^{-1}(g_k^{-1}g_{k+1})/\epsilon=\tau^{-1}(W_k)/\epsilon  
\] 
that can be regarded as the averaged velocity of this
displacement.  The finite difference
$W_k=g_{k}^{-1}\,g_{k+1}\in G$, which in general is an element of a nonlinear space,
can then be represented by the vector $\xi_{k}$. In other words
\begin{equation}\label{llreconstruction}
g_{k+1}=g_k\tau(\epsilon\,\xi_k),
\end{equation}
which may be considered as the {\it discrete reconstruction equation} in the ll scenario. Two standard choices for $\tau$ are the exponential map and the Cayley map (see \cite{Hair} for further details; the latter will be defined for $SO(3)$ in \S\ref{CaySO3}).
The derivative of $\tau$ and its inverse  are defined as follows (see \cite{Bou}):
\begin{definition}\label{Retr}
Given a retraction map $\tau:\mathfrak{g}\Flder G$, its {\rm left trivialized tangent  map} {\rm $\mbox{d}\tau_{\xi}:\mathfrak{g}\Flder\mathfrak{g}$} and the {\rm inverse} {\rm $\mbox{d}\tau_{\xi}^{-1}:\mathfrak{g}\Flder\mathfrak{g}$} of that, are defined such that for $g=\tau(\xi)\in G$ and $\eta,\xi\in\mathfrak{g}$, the following holds
  {\rm \begin{eqnarray*}
    &&\der_{\xi}\tau(\xi)\,\eta=\mbox{d}\tau_{\xi}\,\eta\,\tau(\xi),\\
    &&\der_{\xi}\tau^{-1}(g)\,\eta=\mbox{d}\tau^{-1}_{\xi}(\eta\,\tau(-\xi)).
  \end{eqnarray*}}
\end{definition}
Using these definitions, variations $\delta\xi_k$ and $\delta g_k$ are
related by
\begin{equation}\label{ConsVaria}
\delta \xi_k=\hbox{d}\tau^{-1}_{\epsilon\xi_k}(-\eta_k+\hbox{Ad}_{\tau(\epsilon\xi_k)}
\eta_{k+1})/\epsilon,
\end{equation}
where $\al\ni\eta_k:=g_k^{-1}\delta g_k$; this expression is obtained by
straightforward differentiation of $\xi_k=\tau^{-1}(g_{k}^{-1}\,g_{k+1})/\epsilon$. Note that  $\eta_0=\eta_N=0$ since $\delta g_0=\delta g_N=0$.

 Now, let us consider a Lagrangian function $\tilde l_d:\al\Flder\R$ in order to define a suitable approximation of the reduced action functional $s(\xi)=\int_{t_1}^{t_2}l(\xi(t))dt$  by
\begin{equation}\label{tildeaction}
\tilde s_d=\sum_{k=0}^{N-1}\tilde l_d(\xi_k).
\end{equation}
\begin{remark}
{\rm 
Calling $\tilde l_d:\al\Flder\R$ a   ``discrete Lagrangian''  could be misleading. We are allocating the adjective} discrete {\rm for $L_d:G\times G\Flder\R$ and $l_d:G\Flder \R$, the latter in the reduced case. These are considered as the discrete counterparts of the continuous Lagrangians $L:TG\Flder\R$ and $l:\al\Flder\R$. This is consistent with  the general framework introduced in \cite{Wein} according to which $TG$ and $\al$ are the Lie algebroids associated to the Lie groupoids $G\times G$ and $G$, respectively.   Therefore, we call $\tilde l_d:\al\Flder\R$  just  ``Lagrangian''.}
\end{remark}

Considering a retraction map, the following result is a corollary of theorem \ref{RedTheo}.
\begin{corollary}\label{CoroRetr}
Let $\tilde{l_d}:\al\Flder\R$ be defined as above,  $\tau$ a retraction map and the functions $\tilde{\varphi}^{\alpha}_d:\al\Flder\R$ given by $\tilde{\varphi}^{\alpha}_d:=\varphi^{\alpha}_d\circ\tau$ with $\varphi^{\alpha}_d$ as in \eqref{DiscDist}. Then, the following statements are equivalent:
\begin{enumerate}
\item $\lc \xi_k\rc_{0:N-1}$ is a critical point of the action functional \eqref{tildeaction} with respect to variations $\delta\xi_k$ and a sequence $\lc\eta_k\rc_{0:N}\in\ald$ as itroduced in  \eqref{ConsVaria}.
\item $\lc \xi_k\rc_{0:N-1}$ satisfies the {\rm reduced nonholonomic equations} or {\rm discrete Euler-Poincar\'e-Suslov equations} defined on the Lie algebra:
{\rm
\begin{equation}\label{AltNHEq}
\begin{split}
(\mbox{d}\tau^{-1}_{\epsilon\xi_{k+1}})^{*}\,\tilde l_{d}^{\prime}(\xi_{k+1})-(\mbox{d}\tau^{-1}_{-\epsilon\xi_{k}})^{*}\,\tilde l_{d}^{\prime}(\xi_{k})&=\lambda_{\alpha}\,a^{\alpha},\\ 
\tilde{\varphi}^{\alpha}_d(\xi_{k+1})&=0,
\end{split}
\end{equation}
}
for $k=0,...,N-2$ and where $\lambda_{\alpha}=(\lambda_{k+1})_{\alpha}$ are chosen appropriately.
\end{enumerate}
\end{corollary}
\begin{proof}
By direct computations we obtain
\[
\begin{split}
\delta\sum_{k=0}^{N-1}\tilde l_d(\xi_k)&=\sum_{k=0}^{N-1}\left<\tilde l^{\prime}_d(\xi_k)\,,\,\hbox{d}\tau^{-1}_{\epsilon\xi_k}(-\eta_k+\hbox{Ad}_{\tau(\epsilon\xi_k)}
\eta_{k+1})/\epsilon\right>\\
&=\sum_{k=1}^{N-1}\left<\hbox{Ad}_{\tau(\epsilon\xi_{k-1})}^*(\mbox{d}\tau^{-1}_{\epsilon\xi_{k-1}})^{*}\,\tilde l_{d}^{\prime}(\xi_{k-1})-(\mbox{d}\tau^{-1}_{\epsilon\xi_{k}})^{*}\,\tilde l_{d}^{\prime}(\xi_{k})\,,\,\eta_k/\epsilon\right>=0,
\end{split}
\]
where in the last line we have rearranged the summation index taking into account that $\eta_0=\eta_N=0$. Using that $\hbox{Ad}_{\tau(\epsilon\xi_{k-1})}^*(\mbox{d}\tau^{-1}_{\epsilon\xi_{k-1}})^{*}=(\mbox{d}\tau^{-1}_{-\epsilon\xi_{k-1}})^{*}$ and considering that $\eta_k\in\ald$, we arrive at
\[
(\mbox{d}\tau^{-1}_{-\epsilon\xi_{k}})^{*}\,\tilde l_{d}^{\prime}(\xi_{k})-(\mbox{d}\tau^{-1}_{\epsilon\xi_{k+1}})^{*}\,\tilde l_{d}^{\prime}(\xi_{k+1})=\lambda_{\alpha}a^{\alpha},\quad k=0,...,N-2, 
\]
where we have used the shift $k\mapsto k+1$, and therefore the claim holds.
\end{proof}
Note that  the second equation in \eqref{AltNHEq} defines a subset, which we will denote by $\aldp\subset \al$,  given by the zero level set of $\tilde{\varphi}^{\alpha}_d$, such that in general $\aldp\neq\ald$. Note as well, that the Lagrange multipliers $(\lambda_{k+1})_{\alpha}$ must be chosen appropriately in this case by projecting onto $\aldp$. The equations \eqref{AltNHEq} define a discrete flow $F^{nh}_{\xi}:\aldp\Flder\aldp$, $\xi_{k}\mapsto\xi_{k+1},$ only under some regularity conditions, conditions which may be obtained using the implicit function theorem and which locally amount to the invertibility of the following matrix
\begin{equation}\label{RegCond}
\lp 
\begin{array}{cc}
[\nabla(\mbox{d}\tau^{-1}_{\epsilon\xi_{k+1}})^*)]\tilde l_d^{\prime}(\xi_{k+1})+(\mbox{d}\tau^{-1}_{\epsilon\xi_{k+1}})^{*}\tilde l^{\prime\prime}_d(\xi_{k+1}) & a^{\alpha}\\
\nabla\tilde\varphi^{\alpha}_d(\xi_{k+1}) & 0
\end{array}
\rp.
\end{equation}
Using coordinates we can write the upper-left entry of this matrix as
\[
\begin{split}
[\nabla(\mbox{d}\tau^{-1}_{\epsilon\xi_{k+1}})^{*}]\,\tilde l_{d}^{\prime}(\xi_{k+1})&=\frac{\der(\mbox{d}\tau^{-1}_{\epsilon\xi_{k+1}})_b\,{}^{a}}{\der\xi^c}\,\frac{\der\tilde l_d}{\der\xi^a}(\xi_{k+1}),\\
(\mbox{d}\tau^{-1}_{\epsilon\xi_{k+1}})^{*}\tilde l^{\prime\prime}_d(\xi_{k+1})&=(\mbox{d}\tau^{-1}_{\epsilon\xi_{k+1}})_b\,{}^{a}\,\frac{\der^2\tilde l_d}{\der\xi^c\der\xi^a}(\xi_{k+1}).
\end{split}
\]
Here,  the pull-back of the inverse of the  trivialized  tangent retraction map is a linear operator locally defined by $(\mbox{d}\tau^{-1}_{\epsilon\xi_{k+1}})^{*}:=(\mbox{d}\tau^{-1}_{\epsilon\xi_{k+1}})_b\,{}^{a}$; we shall see that this definition is useful in the case of matrix groups. 

\subsection{Consistency}

One of our main goals in this work is studying the relationship  between the order of consistency of integrators approximating the solutions of \eqref{LdAeqs} and \eqref{EuPoiSus}.  Therefore the following definitions are in order (cf. Remark \ref{RmkMulti}):

\begin{definition}\label{UnredConsis}
By a $(p,s)$ order discretization ((p,s) itegrator) of a nonholonomic problem \eqref{LdAeqs} defined on a Lie group  we understand a sequence of points $\lc(v_{g_k},\lambda_k)\rc_{0:N}\in T_{g_k}G\times\R^{m}$, $\tau_G(v_{g_k})=g_k$, $k=0,...,N$, s.t. 
\[
\begin{split}
&1)\quad|\tau_G(v_{g(t_k+\epsilon)})-\tau_G(v_{g_{k+1}})|\sim O(\epsilon^{r+1}),\\
&2)\quad|T_{g(t_k+\epsilon)}\ell_{g(t_k+\epsilon)^{-1}}v_{g(t_k+\epsilon)}-T_{g_{k+1}}\ell_{g_{k+1}^{-1}}v_{g_{k+1}}|\sim O(\epsilon^{l+1}),
\end{split}
\]
with {\rm min}$(r,l)=p$ and, moreover, {\rm$|\lambda^{{\tiny \mbox{unr}}}(t_k+\epsilon)-\lambda_{k+1}|\sim O(\epsilon^{s+1})$} with $p,s\geq0$. 
\end{definition}
The continuous dynamics ($g(t)$, $v_{g(t)}$, $\lambda^{{\tiny \mbox{unr}}}(t)$) is obtained from the nonholonomic equations \eqref{LdAeqs} (we use the superscript {\it unr} for the multipliers to refer to the {\it unreduced} case.  By $|\cdot - \cdot|$ we will denote (with some abuse of notation) the distance between two elements in several spaces. In this sense, some remarks are in order: 

\begin{itemize}
\item In 1) we are measuring the distance beween points in the Lie group, namely 
\[
|\tau_G(v_{g(t_k+\epsilon)})-\tau_G(v_{g_{k+1}})|=|g(t_k+\epsilon)-g_{k+1}|.
\]
Since $G$ is a nonlinear space in general,  a suitable metric needs to be introduced. We pick a suitable one in the case of $SO(3)$ below.
\item The two vectors $v_{g(t_k+\epsilon)}\in T_{g(t_k+\epsilon)}G$ and $v_{g_{k+1}}\in T_{g_{k+1}}G$  belong to two different vector spaces. Thus, we left translate them to the algebra $\al$, which is a vector space in order to measure their distance. In the case of $\alg$ we will pick the Killing metric to do that.
\item  The multipliers belong to $\R^m$, therefore in this case $|\cdot - \cdot|$ means the usual Euclidean metric.
\end{itemize}
 Now we consider the reduced case.

\begin{definition}\label{RedConsis}
By a $(p,s)$ order discretization ((p,s) integrator)  of the Euler-Poincar\'e-Suslov equations \eqref{EuPoiSus}  we understand a sequence of points $\lc(\xi_k,\lambda_k)\rc_{0:N}\in \al\times\R^{m}$, $k=0,...,N$, s.t. 
{\rm\[
\begin{split}
&1)\quad|\xi(t_k+\epsilon)-\xi_{k+1}|\sim O(\epsilon^{p+1}),\\
&2)\quad|\lambda^{{\tiny\mbox{red}}}(t_k+\epsilon)-\lambda_{k+1}|\sim O(\epsilon^{s+1})
\end{split}
\]}
with $p,s \ge 0$.
\end{definition}
In this case, $|\cdot - \cdot|$ in 1) means a suitable distance in $\al$, while in 2) it is again the Euclidean distance in $\R^m$. The continuous dynamics is determined by the Euler-Poincar\'e-Suslov equations \eqref{EuPoiSus}, and we employ the superscript {\it red} for the multipliers to refer to the  {\it reduced} case. 

Note that Definitions \ref{UnredConsis} and \ref{RedConsis} are completely general, i.e. in principle the  sequences involved need {\it not}  to be $D-$preserving or $\ald-$preserving  in the senses of Definitions \ref{Dpreserv} and \ref{gDpreserv}. 

\begin{proposition}\label{PropoCons1}
A $(p,s)$ integrator of \eqref{LdAeqs} (Definition \ref{UnredConsis}) generates a $(l,s)$ integrator of the Euler-Poincar\'e-Suslov equations \eqref{EuPoiSus} (Definition \ref{RedConsis}).
\end{proposition}
\begin{proof}
Considering the ll symmetry of the problem and left trivialization, we observe that condition 2) in Definition \ref{UnredConsis} implies 
\[
\begin{split}
&|\xi(t_k+\epsilon)-\xi_{k+1}|=|T_{g(t_k+\epsilon)}\ell_{g(t_k+\epsilon)^{-1}}v_{g(t_k+\epsilon)}-T_{g_{k+1}}\ell_{g_{k+1}^{-1}}v_{g_{k+1}}|\sim O(\epsilon^{l+1}).
\end{split}
\]
On the other hand, the relationship between the multipliers in the unreduced and reduced settings is $\lambda^{\tiny\mbox{unr}}(T_g\ell_{g^{-1}}v_g)=\lambda^{\tiny\mbox{red}}(\xi)$, where $\lambda^{\tiny\mbox{unr}}$  are left invariant functions on $TG$ (more precisely on $D$). Therefore
\[
\begin{split}
\lambda^{\tiny\mbox{unr}}(t_k+\epsilon) & =\lambda^{\tiny\mbox{unr}}(v_{g(t_k+\epsilon)})  =\lambda^{\tiny\mbox{unr}}(T_{g(t_k+\epsilon)}\ell_{g(t_k+\epsilon)^{-1}} v_{g(t_k+\epsilon)})\\
 & =\lambda^{\tiny\mbox{red}}(\xi(t_k+\epsilon))=\lambda^{\tiny\mbox{red}}(t_k+\epsilon),
\end{split}
\]
and it follows directly that
\[
|\lambda^{{\tiny \mbox{red}}}(t_k+\epsilon)-\lambda_{k+1}|=|\lambda^{{\tiny \mbox{unr}}}(t_k+\epsilon)-\lambda_{k+1}|\sim O(\epsilon^{s+1}).
\]
This finishes the proof.
\end{proof}
The converse is not trivial since it depends on the discrete reconstruction process \eqref{llreconstruction} and the  metric chosen on $G$. We will consider the particular case of $SO(3)$ and the Cayley map in \S\ref{Consistenci}.

\section{Application to $SO(3)$: the Suslov problem}\label{Principal}
Let us consider the group $SO(3)$ and its  corresponding Lie algebra $\alg$. The latter  is isomorphic to the Euclidean space $\R^3$ through the isomorphism $\hat\cdot:\R^3\Flder\alg$, $\omega\mapsto\hat\omega$, given by
\begin{equation}\label{Isom}
\hat\omega=\lp
\begin{array}{ccc}
0&-\omega_3&\omega_2\\
\omega_3&0&-\omega_1\\
-\omega_2&\omega_1&0
\end{array}
\rp\in\alg
\end{equation}
for $\omega=(\omega_1,\omega_2,\omega_3)\in\R^3$. In this representation, the antisymmetric bracket operation is the standard vector product in $\R^3$ (namely $[\hat\eta,\hat\mu]=\widehat{\eta\times\mu}$ for $\eta,\mu\in\R^3$). 

Furthermore, in the following sections the definition of a distance in $SO(3)$ and a distance in $\alg$ will be relevant, respectively. In the latter case, it is common to pick the Killing form, since it is invariant under all the automorphisms of $\alg$. It is given by
\begin{equation}\label{Killing}
(\hat\xi,\hat\eta)_e=-\frac{1}{2}\mbox{trace}\lp\hat\xi\hat\eta\rp=\frac{1}{2}\mbox{trace}(\hat\xi^T\hat\eta),
\end{equation}
where $\hat\xi,\hat\eta\in\mathfrak{so}(3)$ and $\hat\xi\hat\eta$ represents the usual matrix product. As it is well-known, this bilinear form corresponds to the usual Euclidean product on $\R^3$. Thus, the distance between $\hat\xi,\hat\eta\in\mathfrak{so}(3)$ may be defined by
\begin{equation}\label{distalg}
|\xi-\eta|=|\hat\xi-\hat\eta|=(\hat\xi-\hat\eta,\hat\xi-\hat\eta)_e^{1/2}.
\end{equation}
Concerning the Lie group, to pick a distance is a subtle issue. So far, we have employed $g$ to denote a  point in a general Lie group $G$; from now on $R$ will denote  the points in $SO(3).$ We introduce  the {\it self-distance} in $SO(3)$ as
\begin{equation}\label{Groudist}
\mbox{dist}(R,R)=|I-R\,R^T|,
\end{equation}
where $I$ denotes the identity in $SO(3)$ and we use the Euclidean metric for square matrices, induced by $|R|^2=\sum_{ij}|R_{ij}|^2$. As it is natural, if $R\in SO(3)$, then dist$(R,R)=0$. In other words, what \eqref{Groudist} measures is how {\it far apart}  a matrix is from being orthonormal in terms of the Euclidean metric. This motivates the definition 
\begin{equation}\label{MetMat}
\mbox{dist}(R_1,R_2) :=|I-R_1R_2^T|
\end{equation}
measuring the distance between $R_1$ and $R_2$ within $SO(3)$. It is easy to prove that dist$:SO(3)\times SO(3)\Flder\R$ is indeed a metric.

\subsection{The Cayley map on $SO(3)$}\label{CaySO3}

Concerning the retraction map $\tau$, we employ the Cayley map since it is simple and also  computational e\-ffi\-cient  \cite{Hair}. The Cayley map $\ca:\alg\Flder SO(3)$ is defined by $\ca(\hat\omega)=(I-\frac{\hat\omega}{2})^{-1}(I+\frac{\hat\omega}{2})$. Using the identification \eqref{Isom}, we can find the particular expression
\begin{equation}\label{caypart}
\ca(\hat\omega)=I+\frac{1}{1+|\frac{\omega}{2}|^{2}}\lp\hat\omega+\frac{\hat\omega^{2}}{2}\rp.
\end{equation}
Furthermore,
according to Definition \ref{Retr} it follows
\[
\mbox{d}\ca_{\omega}=\frac{1}{1+|\frac{\omega}{2}|^{2}}(I+\frac{\hat\omega}{2}),\,\,\,\,\,\mbox{d}\ca_{\omega}^{-1}=I-\frac{\hat\omega}{2}+\frac{\omega\,\omega^{T}}{4}.
\]

\subsection{The Suslov problem}\label{SusSO3}

The Suslov problem is a well-known example of a left-left system, introduced in \cite{Suslov}. It describes the motion of a rigid body suspended at one of its points in the presence of a constraint that forces the component of the body angular velocity in a direction fixed in the body frame to vanish. 

Let $\I=(\I_{ij})$ be the inertia tensor $\I:\alg\Flder\alg^*$  of the body, $\I^{-1}=(\I^{ij})$ its inverse, and $\omega\in\R^3$ be the body angular velocity vector. The dynamics is determined through the reduced Lagrangian \eqref{RedLagra}, which in this case reads $l:\alg\Flder\R$
\begin{equation}\label{RedSusLag}
l(\hat\omega)=\frac{1}{2}\bra\I\omega,\omega\ket.
\end{equation}
Let $a$ be the direction, fixed in the body frame, of the vanishing component of the angular velocity. Thus, the nonholonomic constraint reads $\bra a,\omega\ket=0$, where we are using the bracket as the pairing between $(\R^3)^*$ and $\R^3$. We observe that only one constraint is allowed on $\alg$. If there were two independent constraints, then the distribution would be integrable (therefore {\it holonomic}). Without loss of generality, we may choose $a$ as the third component of the body frame $\lc e_1,e_2,e_3\rc$ in $\R^3$, say  $a=e_3=(0,0,1)$. Then, the constraint becomes $\omega_3=0$. Taking into account \eqref{EuPoiSus} and \eqref{RedSusLag}, the Euler-Poincar\'e-Suslov equations, written in $\R^3$, are
\begin{equation}\label{SusPro}
\begin{split}
\I\,\dot\omega&=\I\,\omega\times\omega+\lambda\,e_3,\\
\omega_3&=0,
\end{split}
\end{equation}
where $\lambda\in\R$ is the Lagrange multiplier. Componentwise, we have
\begin{equation}\label{RGB2}
\I\,\lp
\begin{array}{c}
\dot\omega_1\\
\dot\omega_2\\
0
\end{array}
\rp=\lp
\begin{array}{c}
-\omega_2\lp\I_{3i}\omega_i\rp\\
\omega_1\lp\I_{3i}\omega_i\rp\\
\omega_2\lp\I_{1i}\omega_i\rp-\omega_1\lp\I_{2i}\omega_i\rp
\end{array}
\rp+\lambda\,\lp
\begin{array}{c}
0\\
0\\
1
\end{array}
\rp,
\end{equation}
where, from now on, $i=\lc1,2\rc$ (we point out that these equations might be simplified without loss of generality chosing $\I_{12}=0$). After a straightforward computation we arrive at
\begin{equation}\label{ContDyn}
\begin{split}
\dot\omega_1&=-\frac{1}{|\I_m|}\lp \I_{22}\omega_2+\I_{12}\omega_1\rp\lp \I_{3i}\omega_i\rp,\\
\dot\omega_2&=\,\,\,\,\frac{1}{|\I_m|}\lp \I_{21}\omega_2+\I_{11}\omega_1\rp\lp \I_{3i}\omega_i\rp,
\end{split}
\end{equation}
where $\I_m=\lp\begin{array}{cc}
\I_{11}&\I_{12}\\
\I_{21}&\I_{22}
\end{array}\rp$ is  non-degenerate (it is symmetric and positive definite since $\I$ also is); moreover
\begin{equation}\label{LagMult}
\begin{split}
\lambda(\omega)=\omega_1(\I_{2i}\omega_i)&-\omega_2(\I_{1i}\omega_i)\\
&+\frac{\I_{3i}\omega_i}{|\I_m|}\lp(\I_{32}\I_{21}-\I_{31}\I_{22})\omega_2+(\I_{32}\I_{11}-\I_{31}\I_{12})\omega_1\rp.
\end{split}
\end{equation}
In other words, what happens to the Suslov equations \eqref{SusPro} in the pre\-sen\-ce of the nonholonomic constraint $\omega_3=0$ is that they decouple into the differential part \eqref{ContDyn}, i.e. a system of nonlinear ODEs which we will denote by $\dot\omega=f(\omega)$ for simplicity, and the algebraic part \eqref{LagMult}.

\subsection{Order of consistency }\label{Consistenci}
Now we consider the converse statement of Proposition \ref{PropoCons1} for the Suslov problem, where the metric \eqref{MetMat} and the reconstruction equation \eqref{llreconstruction} are given by the Cayley map, i.e. $R_{k+1}=R_k$cay$(\epsilon\hat\omega)$.

\begin{proposition}\label{TeoL}
Consider {\rm $R_{k+1}=R_k$cay$(\epsilon\hat\omega)$} and the metric \eqref{MetMat}. Then a $(p,s)$ integrator of the Suslov problem \eqref{SusPro} (in the sense of Definition \ref{RedConsis}) generates a {\rm ($1,s$)} integrator of the unreduced problem  \eqref{LdAeqs} when $G=SO(3)$ (in the sense of Definition \ref{UnredConsis}). 
\end{proposition}
\begin{proof}
Regarding the multipliers, the argument in the proof of Proposition \ref{PropoCons1} is again valid, i.e.
\[
|\lambda^{{\tiny \mbox{unr}}}(t_k+\epsilon)-\lambda_{k+1}|=|\lambda^{{\tiny \mbox{red}}}(t_k+\epsilon)-\lambda_{k+1}|\sim O(\epsilon^{s+1}).
\]
 Regarding the dynamical part, we observe that the discrete reconstruction equation $R_{k+1}=R_k$cay$(\epsilon\hat\omega)$ now produces the complete sequence $\lc R_k\rc_{0:N}$. Thus, we can left translate the algebra points, which leads to 
\[
\begin{split}
&|T_{g(t_k+\epsilon)}\ell_{g(t_k+\epsilon)^{-1}}v_{g(t_k+\epsilon)}-T_{g_{k+1}}\ell_{g_{k+1}^{-1}}v_{g_{k+1}}|=|\xi(t_k+\epsilon)-\xi_{k+1}|\sim O(\epsilon^{p+1}).
\end{split}
\]
Finally, we need to consider the distance between $R(t_k+\epsilon)$ and $R_{k+1}$ according to \eqref{MetMat}.
Consider the Taylor expansion $R(t_k+\epsilon)=R(t_k)+\epsilon\,\dot R(t_k)+O(\epsilon^2)$. On the other hand,  we have  cay$(\epsilon\hat\omega)=I+\epsilon\hat\omega+\frac{\epsilon^2}{2}\hat\omega^2+O(\epsilon^3)$ according to \eqref{caypart}. This yields  
\begin{equation}\label{DistFinal}
\begin{split}
&|I-R(t_k+\epsilon)R_{k+1}^T|\\
&\,\,\,\,=|I-(R(t_k)+\epsilon\mbox{Ad}_{R(t_k)}\hat\omega(t_k)+O(\epsilon^2))(I-\epsilon\hat\omega+\frac{\epsilon^2}{2}\hat\omega^2+O(\epsilon^3))R_k^T|,
\end{split}
\end{equation}
where we have taken into account the continuous reconstruction equation, i.e. $\dot R=R\hat\omega$ (which in this case actually means $\dot R=\lp\mbox{Ad}_R\hat\omega\rp R$) and the skew-symmetry of $\hat\omega$ when taking the transpose of cay$(\epsilon\hat\omega)$. After imposing  the  initial condition $R(t_k)=R_k$, straightforward computations show that the  terms linear  in $\epsilon$  cancel, while the terms quadratic in $\epsilon$ do not cancel in general, which leads to 
\[
|I-R(t_k+\epsilon)R_{k+1}^T|\sim O(\epsilon^2), \; \mbox{as } \epsilon \to 0.
\]
In other words, the algorithm is consistent of order 1 for  points in $SO(3)$. According to Definition \ref{UnredConsis}, the consistency order of the dynamical part therefore is min($1,p$)$=1$. Thus, the claim of Proposition \ref{TeoL} follows. 
\end{proof}

\subsection{Preservation of the constraints under discretization}\label{dreps} 
The left trivialization is a powerful tool in order to study the preservation of the constraints under discretization for the Suslov problem. Let us consider the initial data $(R_k,\hat\om_k)$ satisfying the nonholonomic constraints \eqref{alD}, i.e. $(\om_k)_3=0$. Consequently, by left trivialization, the associated pair $(R_k,v_{R_k})=(R_k,R_k\hat\om_k)$ also satisfies the original unreduced constraints. Thus, we can construct $R_{k+1}$ by the discrete reconstructin equation \eqref{llreconstruction}, say $R_{k+1}=R_k\tau(\epsilon\hat\om_k)$. Finally, we just need to define $\hat\om_{k+1}$ in terms of the previous data such that it satisfies the constraints, i.e. $\hat\om_{k+1}\in\algd$. This can be done by choosing a suitable discretization of \eqref{SusPro}, namely:
\begin{equation}\label{ DSP}
\begin{split}
\mbox{ DSP}(\hat\om_k,\lambda_{k+1};\hat\om_{k+1})&=0,\\
(\om_{k+1})_3&=0,
\end{split}
\end{equation}
where by  DSP$(\hat\om_k,\lambda_{k+1};\hat\om_{k+1})=0$ we denote a discretization of the first equation in \eqref{SusPro} (the initials are named after Discrete Suslov Problem). In general, this discretization  can be understood as a mapping  DSP$:\alg\times\R\times\alg\Flder\alg^*$. For instance, the variational procedure, through Corollary \ref{CoroRetr}, provides a particular  DSP$(\hat\om_k,\lambda_{k+1};\hat\om_{k+1})=0$ by means of the first equation in \eqref{AltNHEq}, namely
\begin{equation}\label{PartCase}
\mbox{DSP}(\hat\om_k,\lambda_{k+1};\hat\om_{k+1})= (\mbox{d}\tau^{-1}_{-\epsilon\hat\om_{k}})^{*}\,\tilde l_{d}^{\prime}(\hat\om_{k})-(\mbox{d}\tau^{-1}_{\epsilon\hat\om_{k+1}})^{*}\,\tilde l_{d}^{\prime}(\hat\om_{k+1})-\lambda_{k+1}e_3,
\end{equation}
for a general $\tau$. We  denote the coupled discrete equations \eqref{ DSP} by 
\[
\overline{\mbox{DSP}}(\hat\om_k,\lambda_{k+1};\hat\om_{k+1})=0,
\]
and assume them to be regular enough to determine $\hat\om_{k+1}\in\algd$ and $\lambda_{k+1}$ in terms of $\hat\om_k$; particularly, this regularity condition may be described locally, according to the implicit function theorem, by the regularity of the matrix
\begin{equation}\label{REGLOC}
\lp
\begin{array}{cc}
\der_3\mbox{ DSP}(\hat\om_k,\lambda_{k+1};\hat\om_{k+1}) & \der_2\mbox{ DSP}(\hat\om_k,\lambda_{k+1};\hat\om_{k+1})\\
e_3   &   0
\end{array}
\rp, 
\end{equation}
for close enough $\hat\om_k$ and $\hat\om_{k+1}$. Here $\der_i$ denotes the partial derivative with respect to the $i$-th slot in DSP$(\cdot,\cdot;\cdot)$.  This process defines a discrete local flow $(R_k,\hat\om_k)\mapsto(R_{k+1},\hat\om_{k+1})$, and furthermore $(R_k,v_{R_k})\mapsto (R_{k+1},v_{R_{k+1}})$, which schematically leads to the following algorithm.
\begin{algorithm}\label{algo}
\begin{enumerate}
\item[]
\item Input data $(R_k,\hat\om_k)$ s.t. $(\om_k)_3=0,$
\item Set $v_{R_k}=R_k\hat\om_k$,
\item Define $R_{k+1}=R_k\tau(\epsilon\hat\om_k)$,
\item Obtain $\lambda_{k+1}$ and $\hat\om_{k+1}$ from {\rm$\overline{\mbox{ DSP}}(\hat\om_k,\lambda_{k+1};\hat\om_{k+1})=0$} s.t.  $(\om_{k+1})_3=0,$
\item Output data $(R_{k+1},\hat\om_{k+1}, \lambda_{k+1})$,
\item Set $v_{R_{k+1}}=R_{k+1}\hat\om_{k+1}$, Output data $(R_{k+1},v_{R_{k+1}}).$
\end{enumerate}
\end{algorithm}
We note as well that the discrete flow $(R_k,\hat\om_k)\mapsto(R_{k+1},\hat\om_{k+1})$ is well-defined on $SO(3)\times\alg$, also 
$v_{R_k}\mapsto v_{R_{k+1}}$ on $TSO(3)$, for a general $\tau$ and $\epsilon$ small enough.  
\begin{remark}
{\rm
We observe that Algorithm \ref{algo} generates a sequence $\lc\hat\omega_k\rc_{0:N}$ which is $\algd-$preserving in the sense of Definition \ref{gDpreserv}. Moreover, as pointed out, through the process of discrete reconstruction $R_{k+1}=R_k\tau(\epsilon\hat\omega)$ and left trivialization, from $\lc\hat\omega_k\rc_{0:N}$ it also generates a sequence $\lc v_{R_k}\rc_{0:N}$ that is $D-$preserving in the sense of Definition \ref{Dpreserv}. 
}
\end{remark}
On the other hand, the variational procedure described in Theorem \ref{RedTheo} and Corollary \ref{CoroRetr} does {\it not necessarily} provide  $\algd-$preserving integrators. More concretely, the discretization of the constraints provided by \eqref{AltNHEq} reads $\tilde\varphi_d(\hat\om_{k+1})=0$, generating a perturbed constraint set  $\algdp\subset\alg$ in the zero level set of $\tilde\varphi_d$ in general. It follows quite obviously that  $\algdp=\algd$, and thus the $\algd-$preservation is obtained, if we define $\tilde\varphi_d:\alg\Flder\R$ by $\tilde\varphi_d(\cdot)=\bra e_3,\cdot\ket$.  With that choice, through left trivialization we obtain a  $D-$preserving dicretization of the unreduced problem. 
\begin{remark}
{\rm
To generalize the previous conclusions to any ll system on a general Lie group $G$ is not trivial, since we do not have an equivalent of the DSP discretization for the general Euler-Poincar\'e-Suslov equations \eqref{EuPoiSus}. However, the last observatiion can be generalized for the discrete variational scheme \eqref{AltNHEq} with  $\tilde\varphi^{\alpha}_d:=\bra a^{\alpha},\cdot\ket$, yielding the following algorithm:

\begin{algorithm}\label{algoGene}
\begin{enumerate}
\item[]
\item Input data $(g_k,\xi_k)$ s.t. $\bra a^{\alpha},\xi_k\ket=0,$
\item Set $v_{g_k}=g_k\xi_k$,
\item Define $g_{k+1}=g_k\tau(\epsilon\xi_k)$,
\item Under appropriate regularity conditions: Obtain $\lambda_{k+1}$ and  $\xi_{k+1}$ from \eqref{AltNHEq},  such that $\bra a^{\alpha},\xi_{k+1}\ket = 0$,
\item Output data $(g_{k+1},\xi_{k+1}, \lambda_{k+1})$,
\item Set $v_{g_{k+1}}=g_{k+1}\xi_{k+1}$, Output data $(g_{k+1},v_{g_{k+1}}).$
\end{enumerate}
\end{algorithm}
It is easy to see that this algorithm generates a $\ald-$preserving sequence $\lc \xi_k\rc_{0:N}$ and a $D-$preserving sequence $\lc v_{g_k}\rc_{0:N}$.
}
\end{remark}

\medskip

Now we proceed to study some particular discretizations of the Suslov problem.

\subsection{General discretizations}\label{GeneralDisc}

It is interesting to note that any discretization $\overline{\mbox{DSP}}(\omega^k,\lambda_{k+1};\omega^{k+1})=0$\footnote{In this subsection and the next one we are going to raise the index $k$ in the $\om$ variables, to avoid any misleading mixing with the $\R^3$ index, say $\om_i$.} \eqref{PartCase}, respecting the local regularity condition \eqref{REGLOC} and applied to \eqref{RGB2}, leads to particular discretizations of \eqref{ContDyn} and \eqref{LagMult}. Conversely, a general discretization of \eqref{ContDyn} prescribes a DSP of  the Suslov problem, for which we can derive an interesting result. Prior to its statement, we recall that		 $\dot\omega=f(\omega)$ is a system of nonlinear ODEs defined on a vector space, and consequently we can apply any standard numerical method of arbitrary consistency order (Euler, midpoint rule, Runge-Kutta, etc.).  Furthermore, the decoupling of \eqref{RGB2} into differential and algebraic parts allows some freedom. In particular, we choose
\begin{equation}\label{Discretization}
\begin{split}
\I_m\,\lp
\begin{array}{c}
\frac{\omega_1^{k+1}-\omega_1^{k}}{\epsilon}\\\\
\frac{\omega_2^{k+1}-\omega_2^{k}}{\epsilon}
\end{array}
\rp &=\lp
\begin{array}{c}
l_1(\omega^k,\omega^{k+1},\epsilon)\\\\
l_2(\omega^k,\omega^{k+1},\epsilon)
\end{array}
\rp,\\
\lambda_{k+1}&=\lambda(\om^{k+1}),
\end{split}
\end{equation}
where we pick $\lambda_{k+1}$ according to \eqref{LagMult} as a natural choice. We assume $l_1,l_2$ to be smooth functions chosen in such a way that they generate a $p-$th order consistent one-step discretization of \eqref{ContDyn}, i.e. $|\omega(t_k+\epsilon)-\omega^{k+1}|\sim O(\epsilon^{p+1})$, with $p\geq 1$. Namely 
\begin{equation}\label{Discretization2}
\begin{split}
\frac{\omega_1^{k+1}-\omega_1^{k}}{\epsilon}&=\frac{1}{|\I_m|}\lp\I_{22}l_1(\omega^k,\omega^{k+1},\epsilon)-\I_{12}l_2(\omega^k,\omega^{k+1},\epsilon)\rp,\\
\frac{\omega_2^{k+1}-\omega_2^{k}}{\epsilon}&=\frac{1}{|\I_m|}\lp-\I_{21}l_1(\omega^k,\omega^{k+1},\epsilon)+\I_{11}l_2(\omega^k,\omega^{k+1},\epsilon)\rp.
\end{split}
\end{equation}
Now we have all the necessary ingredients to establish the following result.

\begin{proposition}
Any $p-$th order discretization \eqref{Discretization2}, $p\geq 1$, of \eqref{ContDyn} generates a discretization {\rm $\overline{\mbox{DSP}}(\omega^k,\lambda_{k+1}; \omega^{k+1})=0$} \eqref{Discretization} of the Suslov problem \eqref{SusPro}, of order $(p,p)$  (in the sense of Definition \ref{RedConsis}), and consequently at least of order {\rm $(1,p)$} for \eqref{LdAeqs} (in the sense of Definition \ref{UnredConsis}).
\end{proposition}
\begin{proof}
The dynamical part is obvious, namely the order $|\om(t_k+\epsilon)-\om^{k+1}|\sim O(\epsilon^{p+1})$ is given by assumption. On the other hand, given the choice $\lambda_{k+1}=\lambda(\om^{k+1})$ according to \eqref{LagMult} and the $\om$ consistency bound, using the Taylor expansion of $\lambda$ we have
\[
\begin{split}
|\lambda_{k+1}-\lambda(t_k+\epsilon)|&=|\lambda(\om^{k+1})-\lambda(\om(t_k+\epsilon))|\\
&=|\lambda(\om(t_k+\epsilon)+O(\epsilon^{p+1}))-\lambda(\om(t_k+\epsilon))|\\
&=|O(\epsilon^{p+1})\nabla\lambda(\om(t_k))+O(\epsilon^{p+2})|\sim O(\epsilon^{p+1}),
\end{split}
\]
i.e. a $(p,p)$ order discretization of \eqref{SusPro}.
Now, the $(1,p)$ order of  consistency for the unreduced problem follows directly from Proposition \ref{TeoL}.
\end{proof}
\begin{remark}\label{RemarkLambda}
{\rm
The previous result may be refined concerning the algebraic part,  generating a  $(p,s)$ integrator for \eqref{SusPro} with $s>p$. For this purpose, we set
\[
\lambda_{k+1}=\lambda(\om^{k+1})+l_{\lambda}(\om^k,\om^{k+1},\epsilon)
\]
in \eqref{Discretization}, with $l_{\lambda}$ chosen as shown next. In the new scenario, we have (where we omit the $l_{\lambda}$ arguments for the sake of simplicity)
\[
\begin{split}
&\,\,|\lambda_{k+1}-\lambda(t_k+\epsilon)|=|\lambda(\om^{k+1})-\lambda(\om(t_k+\epsilon))+l_{\lambda}|\\
&=|\lambda(\om^{k+1})-\lambda(\om^{k+1}+O(\epsilon^{p+1}))+l_{\lambda}|\\
&=|O(\epsilon^{p+1})D\,\lambda(\om^{k+1})+O(\epsilon^{2p+2})D^2\lambda(\om^{k+1})+O(\epsilon^{3p+3})D^3\lambda(\om^{k+1})+\cdot\cdot\cdot+O(\epsilon^{s+1})+l_{\lambda}|,
\end{split}
\]
where we take the Taylor expansion of $\lambda(\om^{k+1}+O(\epsilon^{p+1}))$ up to the term $\epsilon^{s+1}$ ($s=n\,p+n-1$, with $n$ an integer). Thus, it is obvious that choosing $l_{\lambda}$ such that
\[
O(\epsilon^{p+1})\nabla\lambda(\om^{k+1})+O(\epsilon^{2p+2})\nabla^2\lambda(\om^{k+1})+O(\epsilon^{3p+3})\nabla^3\lambda(\om^{k+1})+\cdot\cdot\cdot+O(\epsilon^{s})+l_{\lambda}=0,
\]
the claim follows.
}
\end{remark}

\subsection{Variational nonholonomic integrators}

As presented in \S\ref{VIRS}, the variational integrator setting provides  a framework for the numerical integration of reduced systems. In particular, corollary \ref{CoroRetr} prescribes a particular DSP, namely,
\begin{equation}\label{VSO3}
\begin{split}
(\mbox{d}\tau^{-1}_{\epsilon\omega^{k+1}})^{*}\,\tilde l_{d}^{\prime}(\omega^{k+1})-(\mbox{d}\tau^{-1}_{-\epsilon\omega^{k}})^{*}\,\tilde l_{d}^{\prime}(\omega^{k})&=\lambda_{k+1}\,e_3,\\ 
\tilde{\varphi}_d(\omega^{k+1})&=0.
\end{split}
\end{equation}
Regarding the discrete constraints $\tilde\varphi_d$, although we have some freedom for their choice, $\tilde\varphi_d(\cdot)=\bra e_3,\cdot\ket$ is the most suitable, since it implies  preservation of the constraint of the reduced system.  In particular, for the Cayley map on $SO(3)$, $\omega_3 = 0$ implies:
\begin{equation}\label{cayinver}
\mbox{d}\ca_{\epsilon\omega}^{-1}=\lp\begin{array}{ccc}
1+\frac{\epsilon^2}{4}\omega_1^2 & \frac{\epsilon^2}{4}\omega_1\omega_2 & -\frac{\epsilon}{2}\omega_2\\
\frac{\epsilon^2}{4}\omega_1\omega_2 & 1+\frac{\epsilon^2}{4}\omega_2^2 & \frac{\epsilon}{2}\omega_1\\
\frac{\epsilon}{2}\omega_2  & -\frac{\epsilon}{2}\omega_1 & 0
\end{array}\rp.
\end{equation}

Setting $\tilde l_d(\hat\omega)=\epsilon\,l(\hat\omega)$ as a first order approximation of the action $s(\hat\omega)=\int_{t_1}^{t_1+\epsilon}l(\hat\omega)\,dt$, where $l:\alg\Flder\R$ is given by \eqref{RedSusLag}, and applying \eqref{VSO3} with $\tilde\varphi_d(\cdot)=\bra e_3,\cdot\ket$, $\tau$ beeing the Cayley map, we obtain the following algorithm for the Suslov Problem:
\begin{equation}\label{taudyn}
\begin{split}
\I_m\lp\begin{array}{c}
\omega_1^{k+1}-\omega_1^{k}\\
\omega_2^{k+1}-\omega_2^{k}
\end{array}\rp&+\frac{\epsilon}{2}\lp
\begin{array}{c}
\,\,\,\omega_2^{k+1}(\I_{3i}\omega_i^{k+1})+\omega_2^{k}(\I_{3i}\omega_i^{k})\\
-\omega_1^{k+1}(\I_{3i}\omega_i^{k+1})-\omega_1^{k}(\I_{3i}\omega_i^{k})
\end{array}
\rp\\
&+\frac{\epsilon^2}{4}\lp
\begin{array}{cc}
(\omega_1^{k+1})^2 &\omega_1^{k+1}\omega_2^{k+1}\\
\omega_1^{k+1}\omega_2^{k+1}& (\omega_2^{k+1})^2
\end{array}
\rp\,\I_m\,\lp\begin{array}{c}
\omega_1^{k+1}\\
\omega_2^{k+1}
\end{array}\rp\\
&-\frac{\epsilon^2}{4}\lp
\begin{array}{cc}
(\omega_1^{k})^2 &\omega_1^{k}\omega_2^{k}\\
\omega_1^{k}\omega_2^{k}& (\omega_2^{k})^2
\end{array}
\rp\,\I_m\,\lp\begin{array}{c}
\omega_1^{k}\\
\omega_2^{k}
\end{array}\rp=0,
\end{split}
\end{equation}

\begin{equation}\label{taualg}
\lambda_{k+1}=\frac{1}{2}\lp \omega_1^{k+1}(\I_{2i}\omega_i^{k+1})+\omega_1^{k}(\I_{2i}\omega_i^{k})\rp-\frac{1}{2}\lp \omega_2^{k+1}(\I_{1i}\omega_i^{k+1})+\omega_2^{k}(\I_{1i}\omega_i^{k})\rp,
\end{equation}
where the rescaling $\lambda_{k+1}\mapsto -\lambda_{k+1}/\epsilon^2$ has been introduced. This rescaling may be understood in the context of the construction of the variational integrator using a discretization map $\rho:SO(3)\times SO(3)\Flder TSO(3)$. More concretely, it can be shown \cite{JiSch} that  any discretization $\phi\propto \mu\circ\rho$ preserves the constraint for the discrete flow $(\tilde R_k,\dot{\tilde{R}}_k)\mapsto (\tilde R_{k+1},\dot{\tilde{R}}_{k+1})$, not only $\phi= \mu\circ\rho$ as stated in Proposition \ref{Preservation} (note that we consider the redefinition $\tilde R_k$ of the discrete nodes as prescribed as well by that proposition). Thus, setting $\phi=\frac{-1}{\epsilon ^2} \mu\circ\rho$ accounts for the mentioned rescaling.

Concerning the order of consistency w.r.t. the continuous Suslov problem of the discretization prescribed by the algorithm in \eqref{taudyn} and  \eqref{taualg}, we prove the following result.
\begin{proposition}\label{PropoFinal}
The numerical method \eqref{taudyn}, \eqref{taualg}, is consistent of order  $(2,\star)$  (Definition \ref{RedConsis}) with respect to the Suslov problem, and consequently of order $(1,\star)$ (Definition \ref{UnredConsis}) with respect to \eqref{LdAeqs}. Here by $\star$ we mean that the method is {\rm not} neccessarily consistent  concerning the multipliers. 
\end{proposition}
\begin{proof}
To prove this result, we just consider the Taylor expansion $\omega(t_k+\epsilon)=\omega(t_k)+\epsilon\,\dot\omega(t_k)+\frac{\epsilon^2}{2}\ddot\omega(t_k)+O(\epsilon^3)$, where $\dot\omega$ is given in \eqref{ContDyn}, and compare, order by order, with $\omega^{k+1}$ provided by \eqref{taudyn}. In first place, we calculate the second order time derivatives, namely
\[
\begin{split}
\ddot\omega_1=&-\frac{1}{|\I_m|^2}\lc \I_{22}(\I_{3i}\omega_i)^2(\I_{21}\omega_2+\I_{11}\omega_1)-\I_{12}(\I_{3i}\omega_i)^2(\I_{22}\omega_2+\I_{12}\omega_1)\rc\\
&-\frac{1}{|\I_m|^2}\,\I_{22}\,\omega_2\,(\I_{3i}\omega_i)\lc-\I_{31}(\I_{22}\omega_2+\I_{12}\omega_1)+\I_{32}(\I_{22}\omega_2+\I_{12}\omega_1)\rc\\
&-\frac{1}{|\I_m|^2}\,\I_{12}\,\omega_1(\I_{3i}\omega_i)\lc-\I_{31}(\I_{22}\omega_2+\I_{12}\omega_1)+\I_{32}(\I_{22}\omega_2+\I_{12}\omega_1)\rc,
\end{split}
\]
and
\[
\begin{split}
\ddot\omega_2=&\,\,\,\,\,\,\,\frac{1}{|\I_m|^2}\lc \I_{21}(\I_{3i}\omega_i)^2(\I_{21}\omega_2+\I_{11}\omega_1)-\I_{11}(\I_{3i}\omega_i)^2(\I_{22}\omega_2+\I_{12}\omega_1)\rc\\
&+\frac{1}{|\I_m|^2}\,\I_{21}\,\omega_2\,(\I_{3i}\omega_i)\lc-\I_{31}(\I_{22}\omega_2+\I_{12}\omega_1)+\I_{32}(\I_{21}\omega_2+\I_{11}\omega_1)\rc\\
&+\frac{1}{|\I_m|^2}\,\I_{11}\,\omega_1(\I_{3i}\omega_i)\lc-\I_{31}(\I_{22}\omega_2+\I_{12}\omega_1)+\I_{32}(\I_{21}\omega_2+\I_{11}\omega_1)\rc.
\end{split}
\]
On the other hand, the first two terms in \eqref{taudyn} imply that $\omega^{k+1}=\omega^k+O(\epsilon)$ (ensuring consistency) and, moreover, that 
\begin{equation}\label{forder}
\lp\begin{array}{c}
\omega_1^{k+1}\\
\omega_2^{k+1}
\end{array}
\rp=\lp\begin{array}{c}
\omega_1^{k}\\
\omega_2^{k}
\end{array}
\rp+\epsilon\,\I_m^{-1}\lp
\begin{array}{c}
-\omega_2^k(\I_{3i}\omega_i^k)\\
\omega_1	^k(\I_{3i}\omega_i^k)
\end{array}
\rp,
\end{equation}
which implies first order consistency in view of  \eqref{ContDyn}. Futhermore, this implies that the third and fourth term in \eqref{taudyn} cancel out at $O(\epsilon^2)$ order; thus, we 	realize that the relevant term in \eqref{taudyn} at this order is just
\begin{equation}\label{sorder}
\frac{\epsilon^2}{2}\,\I_m^{-1}\,\lp\begin{array}{c}
-\omega_2^{k+1}(\I_{3i}\omega_i^{k+1})\\
\omega_1^{k+1}(\I_{3i}\omega_i^{k+1})
\end{array}
\rp.
\end{equation}
Using \eqref{forder}, we obtain
\[
\begin{split}
-\omega_2^{k+1}(\I_{3i}\omega_i^{k+1})&=-\frac{1}{|\I_m|}(\I_{3i}\omega_i^k)^2(\I_{21}\omega_2^k+\I_{11}\omega_1^k)\\
&-\frac{1}{|\I_m|}\omega_2^k(\I_{3i}\omega_i^k)\lc-\I_{31}(\I_{22}\omega_2^k+\I_{12}\omega_1^k)+\I_{32}(\I_{21}\omega_2^k+\I_{11}\omega_1^k)\rc+O(\epsilon),
\end{split}
\]
\[
\begin{split}
\omega_1^{k+1}(\I_{3i}\omega_i^{k+1})&=-\frac{1}{|\I_m|}(\I_{3i}\omega_i^k)^2(\I_{22}\omega_2^k+\I_{12}\omega_1^k)\\
&+\frac{1}{|\I_m|}\omega_1^k(\I_{3i}\omega_i^k)\lc-\I_{31}(\I_{22}\omega_2^k+\I_{12}\omega_1^k)+\I_{32}(\I_{21}\omega_2^k+\I_{11}\omega_1^k)\rc+O(\epsilon).
\end{split}
\]
Plugging these terms into \eqref{sorder}, we obtain 
\[
\begin{split}
&-\frac{(\epsilon^2/2)}{|\I_m|^2}\lc \I_{22}(\I_{3i}\omega_i^k)^2(\I_{21}\omega_2^k+\I_{11}\omega_1^k)-\I_{12}(\I_{3i}\omega_i^k)^2(\I_{22}\omega_2^k+\I_{12}\omega_1^k)\rc\\
&-\frac{(\epsilon^2/2)}{|\I_m|^2}\,\I_{22}\,\omega_2^k\,(\I_{3i}\omega_i^k)\lc-\I_{31}(\I_{22}\omega_2^k+\I_{12}\omega_1^k)+\I_{32}(\I_{22}\omega_2^k+\I_{12}\omega_1^k)\rc\\
&-\frac{(\epsilon^2/2)}{|\I_m|^2}\,\I_{12}\,\omega_1^k(\I_{3i}\omega_i^k)\lc-\I_{31}(\I_{22}\omega_2^k+\I_{12}\omega_1^k)+\I_{32}(\I_{22}\omega_2^k+\I_{12}\omega_1^k)\rc,
\end{split}
\]
and
\[
\begin{split}
&\,\,\,\,\,\,\frac{(\epsilon^2/2)}{|\I_m|^2}\lc \I_{21}(\I_{3i}\omega_i^k)^2(\I_{21}\omega_2^k+\I_{11}\omega_1^k)-\I_{11}(\I_{3i}\omega_i^k)^2(\I_{22}\omega_2^k+\I_{12}\omega_1^k)\rc\\
&+\frac{(\epsilon^2/2)}{|\I_m|^2}\,\I_{21}\,\omega_2^k\,(\I_{3i}\omega_i^k)\lc-\I_{31}(\I_{22}\omega_2^k+\I_{12}\omega_1^k)+\I_{32}(\I_{21}\omega_2^k+\I_{11}\omega_1^k)\rc\\
&+\frac{(\epsilon^2/2)}{|\I_m|^2}\,\I_{11}\,\omega_1^k(\I_{3i}\omega_i^k)\lc-\I_{31}(\I_{22}\omega_2^k+\I_{12}\omega_1^k)+\I_{32}(\I_{21}\omega_2^k+\I_{11}\omega_1^k)\rc.
\end{split}
\]
Now, substraction from the expressions for $\ddot\omega_1,\ddot\omega_2$ presented above, we obviously have $|\omega(t_k+\epsilon)-\omega^{k+1}|\sim O(\epsilon^3)$. Furthermore, the factors $\frac{1}{4}$  in \eqref{taudyn} prevent the $O(\epsilon^3)$ terms on the discrete and continuous sides to coincide. Regarding the multipliers, it is straightforward to see from \eqref{taualg} that $\lambda_{k+1}= \omega_1^{k}(\I_{2i}\omega_i^{k})- \omega_2^{k}(\I_{1i}\omega_i^{k})+O(\epsilon)$, while $\lambda(t_k+\epsilon)=\lambda(\omega^k)+O(\epsilon)$, where $\lambda(\omega)$ is determined by \eqref{LagMult}. Hence, we see that $\lambda(t_k+\epsilon)-\lambda_{k+1}=O_0+O(\epsilon)$ with 
\[
O_0=\frac{\I_{3i}\om_i^k}{|\I_m|}\lp(\I_{32}\I_{21}-\I_{31}\I_{22})\om_2^k+(\I_{32}\I_{11}-\I_{31}\I_{12})\om_1^k\rp,
\]
which is different from zero, in general, and consequently the discrete multiplier is not consistent with the continuous one. This shows the first claim. Concerning the second, it suffices to apply Proposition \ref{TeoL}.
\end{proof}
Therefore, the variational integrator setting generates a second-order consistent method on the {\it dynamical} side (a fact which is interesting, since we are considering a first-order consistent discrete Lagrangian $\tilde l_d$, and therefore we might expect a first-order consistent numerical method in the spirit of \cite{MarsdenWest}) while it is not neccessarily consistent on the {\it algebraic} side. Needless to say, this is a drawback of the numerical scheme. However, due to the decoupling between the two parts mentioned above (we can obtain the $\omega^{k+1}$ values independently of the $\lambda_{k+1}$'s), we may perform, besides the $\epsilon^2$ rescaling, a discrete shift of $\lambda(\omega^{k+1})$ as described in remark \ref{RemarkLambda}, generating in consequence a $(2,s)$ method for \eqref{SusPro}, and therefore a ($1,s$) method for \eqref{LdAeqs}. Unfortunately, such a shift cannot be understood in general as an alternate choice of the discretization map $\rho$, if this map is not just a rescaling  but involves nonlinear terms depending on the $\om$ variables.

\medskip

We illustrate these facts by the following discussion and plots. Before going into details, we point out that the reduced energy $E_l(\xi)=\bra\frac{\der l}{\der\xi},\,\xi\ket-l(\xi)\Big|_{\xi\in\dal}$ is preserved  along the solutions of the Euler-Poincar\'e-Suslov equations (see \cite{FeZen}). In our case, the reduced energy along the solutions reads $E_l(\hat\omega)=\frac{1}{2}\bra\I\hat\omega,\hat\omega\ket=\omega_1(\I_{1i}\omega_i)+\omega_2(\I_{2i}\omega_i)$ (with $i=\lc1,2\rc$). The preservation of this energy should be taken into account as a favorable property  of nonholonomic integrators, as shown below.

We consider a homogeneous rigid body with inertia matrix $\I=\lp\begin{array}{ccc}
1&0.1&0.2\\
0.1&1&0.2\\
0.2&0.1&1
\end{array}\rp$ and initial values $\omega_1(0)=0.4$ and $\omega_2(0)=0.5$ (w.r.t. proper unities). We shall display the performance of an order (2,1) $\overline{\mbox{DSP}}$ of type \eqref{Discretization}, and the variational nonholonomic integrator \eqref{taudyn} and \eqref{taualg}, which is also order 2 w.r.t.  the dynamical variables as proved in proposition \ref{PropoFinal}. More concretely, the integrator $\overline{\mbox{DSP}}$ corresponds to the midpoint rule, namely we have  $l_1(\omega^k,\omega^{k+1})=-\lp\frac{\omega_2^{k+1}+\omega_2^{k}}{2}\rp(\I_{3i}\lp\frac{\omega_i^{k+1}+\omega_i^{k}}{2}\rp)$ and
$l_2(\omega^k,\omega^{k+1})=\lp\frac{\omega_1^{k+1}+\omega_1^{k}}{2}\rp(\I_{3i}\lp\frac{\omega_i^{k+1}+\omega_i^{k}}{2}\rp)$ in \eqref{Discretization}, whe\-re we recall that $i=\lc1,2\rc$ (note as well that in this case there is no $\epsilon$ dependence of $l_1,l_2$). In order to achieve order 2 consistency with respect to \eqref{LagMult}, we determine $\lambda_{k+1}$ according to \eqref{Discretization}.

On the other hand, the variational nonholonomic integrator \eqref{taudyn} and \eqref{taualg} corresponds to the setup $\tilde l_d(\hat\omega)=\epsilon\,l(\hat\omega)$, $\tilde\varphi_d(\hat\omega)=\omega_3$ and the retraction map $\tau=$ cay. What we observe is that, for small time steps $\epsilon$, the behavior of both integrators is indistinguishable except with respect to the multipliers. As to be expected due to proposition \ref{PropoFinal}, the nonholonomic variational integrator produces an inconsistent  discretization of the Lagrange multipliers. However, we display also the performance of both integrators over a big time interval, noticing that here the variational integrator's performance is much better, mainly with respect to the preservation of energy, where we observe a fast decay in case of the midpoint rule. Moreover, we observe that the variational integrator even works quite well if the step size of the time discretization is taken to be relatively big.

\begin{figure}[H]
\begin{minipage}{.49\linewidth}
\hspace{-4cm} \includegraphics[width=142mm,height=160mm]{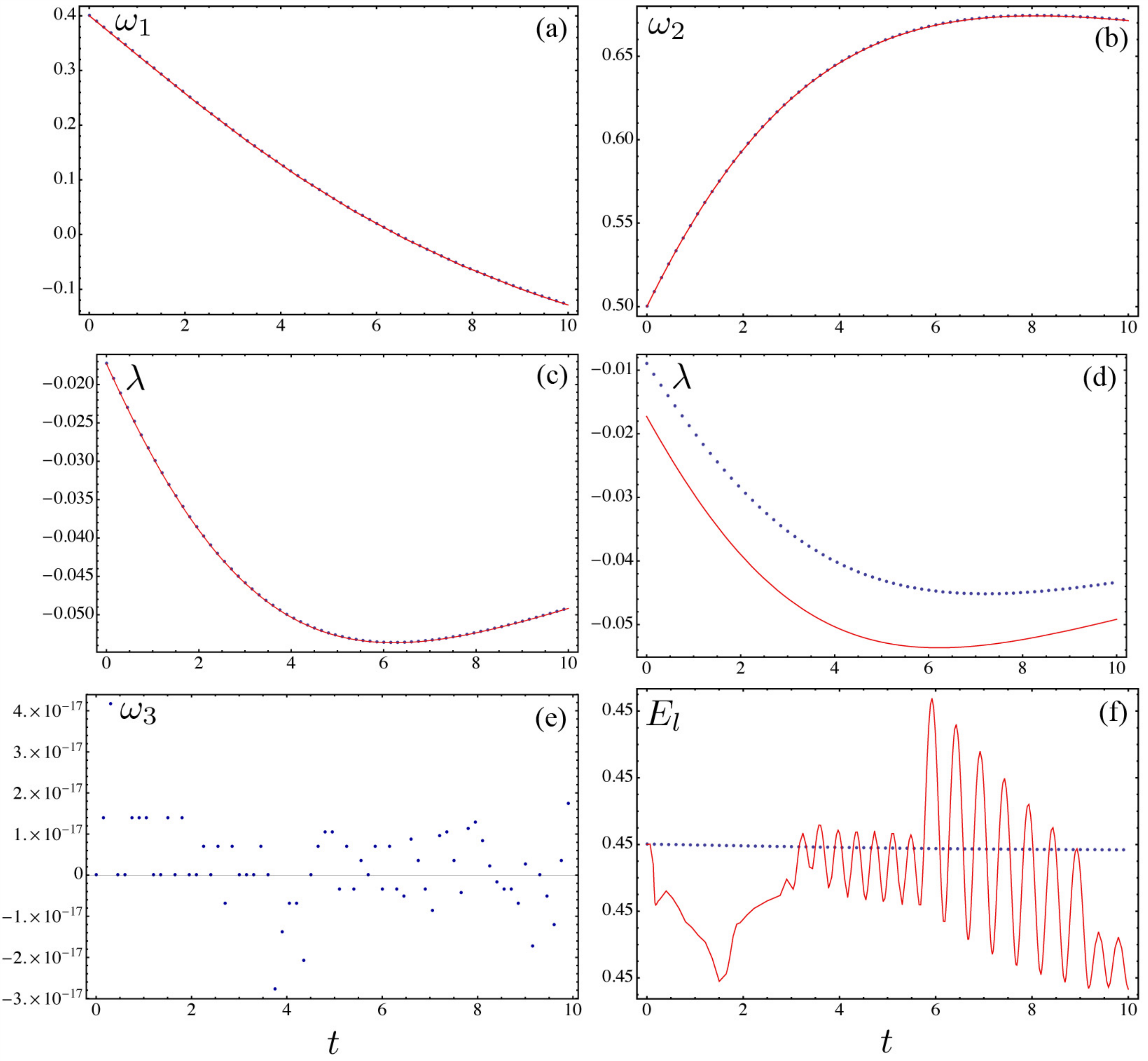}
\end{minipage}
\caption{In this figure we display the performance of the midpoint rule ($\overline{\mbox{DSP}}(\omega^k,\lambda_{k+1};\omega^{k+1})=0$, with inertia matrix $\I$ and initial values $\omega_1(0)$ and $\omega_2(0)$ introduced above) for the nonholonomic rigid body with a time step of size  $\epsilon=10^{-3}$. The solid red line is obtained through a RK4 integrator (which we consider an accurate approximation of the continuous nonlinear dynamics over a short time interval), while the blue dots represent the performance of the midpoint rule. The plots $(a)$ and $(b)$ correspond to the dynamical variables $\omega_1$, $\omega_2$, while $(c)$ displays the Lagrange multipliers $\lambda.$ On the other hand $(d)$ shows the inconsistent multipliers generated by the nonholonomic variational integrator. Finally, $(e)$ and $(f)$ show the preservation of the constraints and the energy $E_l(\hat\omega)$ up through round off errors, respectively.}
\label{Fig1}
\end{figure}

\begin{figure}[H]
\begin{minipage}{.49\linewidth}
\hspace{-4cm} \includegraphics[width=142mm,height=160mm]{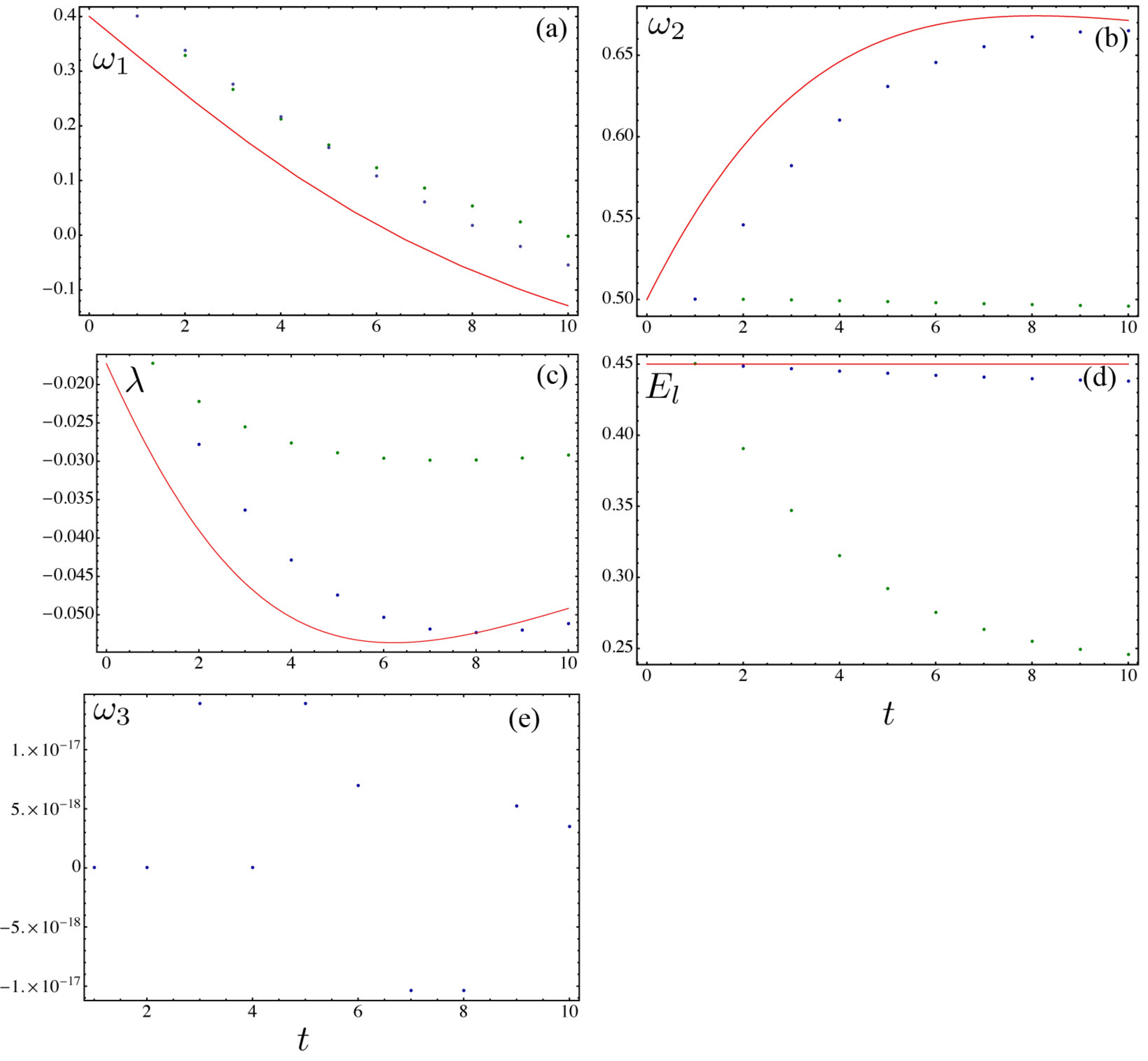}
\end{minipage}
\caption{This figure displays the comparison between the midpoint rule (the same as in Figure \ref{Fig1}) and the variational integrator \eqref{taudyn}, \eqref{taualg}, for a time step of size  $\epsilon=10^0=1$ (we recall that this integrator is also order 2 consistent in the dynamical variables). The former is represented by the green points and the latter by the blue ones, while the solid red line still represents the performance of a RK4 method. Variables $\omega_1$ $(a)$, $\omega_2$ $(b)$, $\lambda$ $(c)$ and $E_l$ $(d)$ are displayed, while $(e)$ shows the preservation of the constraints by the variational integrator up through  round off errors. We observe a better performance of the variational integrator, mainly with respect to the preservation of energy, a fact which, considering  bigger time steps, leads to the conclusion that its convergence to the actual solution is much faster and its long-term behavior is much more accurate.}
\label{Fig2}
\end{figure}

\section{Discretization of the Euler-Poincar\'e-Suslov problem as perturbation}\label{Pertur}
So far, we have focused on the discretization of the Suslov problem, paying attention to its consistency order and $\algd-$pre\-ser\-va\-tion. In this section we return to any Lie group $G$ and its Lie algebra $\al$, since the relevant results from above apply in this general case as well; in the last part we will particularize $G=SO(3)$.

As it is well-known, by discretizing the dynamics we introduce some discrepancies with respect to the continuous system even when the discretization is performed
in some kind of structure-preserving fashion. It is needless to mention that this is a
central and fundamental question for all kinds of numerical investigations, especially
concerning the long-term evolution of dynamical systems. Regarding this issue,
we refer to \cite{FiSch} where a positive answer to the following question is given: Is it
possible to embed a numerical scheme approximating the continuous-time 
flow of
a set of autonomous ordinary differential equations (ODE) into the time evolution
corresponding to a non-autonomous perturbation of the original autonomous ODE?
The positive result may be phrased as: {\it Any $p-$th order discretization of an autonomous ODE
can equivalently be viewed as the time$-\epsilon$ period map of a suitable $\epsilon-$periodic non-autonomous perturbation of the original ODE (where $\epsilon$ is the fixed step size of the discretization).}  The precise statement is:

\begin{theorem}\label{TheoPer}
Suposse that $h\in C^r(\R^n,\R^n)$, $r\geq 1$, and consider the autonomous ODE
\begin{equation}\label{Eq1}
\dot x=h(x).
\end{equation}
Let $F(t,x)$ be the  solution flow of \eqref{Eq1} satisfying $F(0,x)=0,$ and assume that there are an integer $\rho\geq 1$, a continuous function $C:[0,\infty)\Flder [0,\infty)$ and a one-step difference approximation of step size $\epsilon$
\[
x_{k+1}=\phi(\epsilon,x_k),\,\,\, (0<\epsilon\leq\epsilon_0; \, k\in\Z)
\]
which is consistent of order $p$, i.e.
\[
|\phi(\epsilon,x)-F(\epsilon,x)|\leq C(|x|)\,\epsilon^{p+1}.
\]
Then, there exists a function $d(\epsilon,t/\epsilon,x)$, as smooth as $h$ and periodic in $t$ of period $\epsilon$, such that if $G(t,s;\epsilon,x)$\footnote{Do not confuse with the Lie group $G$.}, $G(s,s;\epsilon,x)=x$, is the  solution flow of the non-autonomous, $\epsilon-$periodic ODE
\begin{equation}\label{Eq2}
\dot x=h(x)+\epsilon^pd(\epsilon,t/\epsilon,x),
\end{equation}
then
\[
G(\epsilon,0;\epsilon,x)=\phi(\epsilon,x),
\]
where $G(\epsilon,0;\epsilon,\cdot):\R^n\Flder\R^n$ is the Poincar\'e map (period map) for \eqref{Eq2}, corresponding to initial time $s=0$.
\end{theorem}   
Following these ideas, in \cite{JiSch} the case of nonholonomic dynamics in $\R^n$ has been explored, and the following theorem has been  proved. In the context of the present work, i.e. when the configuration manifold is a Lie group, it has to be understood w.r.t. coordinates.  
\begin{theorem}\label{Per11}
Let the matrix $\lp\frac{\der^2L}{\der v^i_g\der v^j_g}\rp$ be  positive-definite. Then, any $D-$preserving    discretization of the nonholonomic equations \eqref{LdAeqs}, the consistency order of which is $p$ w.r.t. the dynamical variables,  can be embedded into the time evolution of a non-autonomous perturbation of  the following form: 
\begin{equation}\label{PerFin1}
\begin{split}
\dot g^i&=v^i_g+\epsilon^p\tilde d^i_g(\epsilon,t/\epsilon,g,v_g),\\
\frac{\der^2L}{\der v^i_g\der v^j_g}\dot v^j_g&=\frac{\der L}{\der g^i}+\lambda_{\alpha}\,\mu^{\alpha}_i(g)+\epsilon^p(\tilde d_{v_g})_i(\epsilon,t/\epsilon, g,v_g),\\
\mu^{\alpha}_i(g)v^i_g&=0.
\end{split}
\end{equation}
\end{theorem}
An analogous result can be obtained in the reduced setting. There are some differences though; first, we note that the differential equations in \eqref{EuPoiSus} are of order 1 rather than of order 2 as  in \eqref{LdAeqs}; second, the constraints $\bra a^{\alpha},\xi\ket=0$ determine a linear subspace $\ald\subset\al$  instead of a regular submanifold. The procedure to obtain the perturbation of the nonholonomic dynamics produced by a given discretization can be split into three steps here:
\begin{enumerate}
\item Define an ODE evolving on $\ald$ from the Euler-Poincar\'e-Suslov equations \eqref{EuPoiSus} by projection, which we will call Euler-Poincar\'e-Suslov ODE;
\item Apply Theorem \ref{TheoPer};
\item Undo the projection process to recover the perturbed Euler-Poincar\'e-Suslos equations.
\end{enumerate}
Thus, we obtain the following result. 
\begin{theorem}\label{PerTh}
Let $\lp\frac{\der^2l}{\der\xi^a\der \xi^b}\rp$ be a  positive-definite matrix. Any $\ald-$preserving discretization of the Euler-Poincar\'e-Suslov equations \eqref{EuPoiSus}, the consistency order of which is $p$ w.r.t. the dynamical part, can be embedded into the time evolution of a non-autonomous perturbation of  the following form
\begin{equation}\label{PerFin2}
\begin{split}
\frac{\der^2l}{\der \xi^b\der\xi^c}\dot\xi^c&=C_{be}^d\xi^e\frac{\der l}{\der\xi^d} +\lambda_{\alpha}a^{\alpha}_b + \epsilon^p\tilde d_b(\epsilon, t/\epsilon,\xi),\\
\bra a^{\alpha},\xi\ket&=0,
\end{split}
\end{equation}
where  $C_{ab}^e$ denote the structure constants of $\al$ w.r.t. the coordinates chosen.
\end{theorem}
For convenience, we outline the proof in the appendix. Of course, it would be ideal, if the perturbed Euler-Poincar\'e-Suslov
equations in \eqref{PerFin2} admitted a Lagrangian structure of some kind. The question, when
this is true and how this is related to specific properties of both, the underlying
discretization of the unperturbed problem as well as the selected type of embedding, is left as a subject of further study.

\section{Conslusions}
We followed the ideas presented in \cite{JiSch} to study the discretization of the Suslov problem. First, concerning the  order of consistency of  discretizations corresponding to the unreduced and reduced settings and related to each other by reduction and reconstruction, respectively, we found that when the discrete reconstruction equation is given by a Cayley retraction map both consistency orders are related to each other, too. The order of consistency carries over unchanged from the unreduced to the reduced setting. It becomes zero in the unreduced setting, no matter how big it is in the reduced setting.  Furthermore, we studied distribution preserving integrators,  showing that this property may be achieved for general numerical schemes. We presented a specific algorithm with that property  based on the general reduced framework.  We considered two examples of different integrators, one of them based on the midpoint rule while the other is based on a variational scheme. As proved, both are order-2 consistent in the dynamical variables, but it is numerically shown that the variational one converges faster to the actual solution and shows a better long-term behaviour. Finally, concerning the discretization understood as perturbation, we proved that any distribution preserving integrator of the Euler-Poincar\'e-Suslov equations (in general)
may be understood as a non-autonomous perturbation of the continuous dynamics.

\medskip\medskip

{\bf Acknowledgments}: We are indebted to the referee for the constructive comments and corrections, which helped a lot to substantially improve the manuscript. 	We thank Carlos Navarrete-Benlloch for his help in the display of numerical results, and also Dimtry Zenkov and Yuri Fedorov for helpful comments during the ``Workshop on Nonholonomic mechanics and optimal control'', held at the Institute Henri Poincar\'e, Paris,  November 2014. The first author is  indebted to Luis Garc\'ia-Naranjo for fruitful discussions during their common visit to Technische Universit\"at Berlin. Finally we thank the Deutsche Forschungsgemeinschaft (DFG) for supporting this research within  the frame of the project B4 of SFB-Transregio 109
``Discretization in Geometry and Dynamics''.

\section*{Appendix: Sketch of the Proof of Theorem \ref{PerTh}}

We first introduce some notation. We set $m\equiv(m_{ab}):=\lp\frac{\der^2l}{\der\xi^a\der \xi^b}\rp$, while $(m^{ab})\equiv m^{-1}$ denotes its inverse (recall that $(m_{ab})$ is regular and positive-definite). With this, the first equation in \eqref{EuPoiSus} may be rewritten as
\begin{equation}\label{EuPoiSusInv}
\dot\xi^b=f^b(\xi)+\lambda_{\alpha}m^{bc}(a^{\alpha})_c,
\end{equation}
where $f^b(\xi):=m^{bc}C^d_{ce}\xi^e\frac{\der l}{\der\xi ^d}.$
\\

(1) Since $\mu^{\alpha}(g)=\ell^*_{g^{-1}}a^{\alpha}$ represent a set of linearly independent one-forms spanning $D_g^{\circ}\subset T^*_gG$, it is easy to see that the $a^{\alpha}$ spanning $(\ald)^{\circ}\subset \al^*$ are also linearly independent. Therefore,  we can decompose the algebra as 
\begin{equation}\label{Decomp}
\al=\ald\oplus(\ald)^{\perp},
\end{equation}
with respect to the metric represented by $(m_{ab})$. Thus, from \eqref{EuPoiSusInv} we can derive an ODE  on $\ald$ by eliminating the Lagrange multipliers. More precisely, if we take the time derivative of the nonholonomic constraints we obtain $\bra a^{\alpha},\dot\xi\ket=0$ (note that $a^{\alpha}$ are constant), which, after replacing $\dot\xi$ by the right hand side of the equation \eqref{EuPoiSusInv}, yields
\begin{equation}\label{LagElim}
\lambda_{\alpha}(\xi)=-\mathcal{C}_{\alpha\beta}\bra a^{\beta},f(\xi)\ket,
\end{equation}
where $\lp\mathcal{C}^{\alpha\beta}\rp:=\lp\bra a^{\alpha},m^{-1}a^{\beta}\ket\rp$ and $\lp\mathcal{C}_{\alpha\beta}\rp=\mathcal{C}^{-1}.$ It is possible to prove the in\-ver\-ti\-bi\-li\-ty of $\mathcal{C}$ using geometric arguments (see for instance \cite{JiSch}, lemma 2.4), or just by arguing that $m$ is of full rank and $a^{\alpha}$ of constant rank. Thus, we obtain the Euler-Poincar\'e-Suslov ODE  on $\ald$ given by
\begin{equation}\label{EuPoiSusODE}
\dot\xi=h(\xi),
\end{equation}
with $\bra a^{\alpha},\xi\ket=0$, where $h(\xi):=f(\xi)+\lambda_{\alpha}(\xi)\,m^{-1}a^{\alpha}$ and $\lambda_{\alpha}(\xi)$ is defined in \eqref{LagElim}.

(2) Now, let us choose adapted coordinates with respect to the decomposition \eqref{Decomp}, say 
$\xi = (\xi^{a}) = (\xi^{\bar a}, \xi^{\hat a})$, where $a=1,...,n$, $\bar a=1,...,n-m$ (corresponding to $\ald$) and $\hat a=n-m+1,...,n$ (corrresponding to $(\ald)^{\perp}$). Since $\al$ is a linear space, the choice of such adapted coordinates is trivial. Projecting \eqref{EuPoiSusODE} onto $\ald$, we obtain 
\[
\dot\xi^{\bar a}= h^{\bar a}(\bar\xi),
\]
where we write $\bar \xi$ for $(\xi^{\bar a})$.  Now, we can apply theorem \ref{TheoPer} to $\dot\xi^{\bar a}= h^{\bar a}(\bar\xi)$, ensuring that any $p-$th order discretization can be
viewed as the time$-\epsilon$ map of a suitable non-autonomous perturbation, which is  $\epsilon-$periodic in $t$, namely
\begin{equation}\label{Per1}
\dot\xi^{\bar a}= h^{\bar a}(\bar\xi)+\epsilon^pd^{\bar a}(\epsilon,t/\epsilon,\bar\xi).
\end{equation}

(3) This basically finishes the proof, since this equation can be derived from  perturbed Euler-Poincar\'e-Suslov equations of the form
\begin{equation}\label{Per3}
\begin{split}
\dot\xi&=f(\xi) +\lambda_{\alpha}\,m^{-1}a^{\alpha} + \epsilon^p\,m^{-1}\,\tilde d(\epsilon, t/\epsilon,\xi),\\
\bra a^{\alpha},\xi\ket&=0.
\end{split}
\end{equation}
Indeed, we see that the first equation in \eqref{Per3} leads to \eqref{Per1} by eliminating the Lagrange multipliers, which gives
\[
\lambda_{\alpha}=\lambda_{\alpha}(\xi)-\epsilon^p\mathcal{C}_{\alpha\beta}\bra a^{\beta},\,m^{-1}\tilde d\ket,
\]
where $\lambda_{\alpha}(\xi)$ is defined in \eqref{LagElim}. Plugging this into \eqref{Per3} we finally obtain \eqref{Per1}, including as well the identification $d = (d^{\bar a}) =m^{-1}\tilde d-\mathcal{C}_{\alpha\beta}\,m^{-1}\,a^{\alpha}\,\bra a^{\beta},m^{-1}\tilde d\ket$.

This finishes the proof.

\end{document}